\documentclass[12pt]{amsart}
\usepackage{amscd,amssymb,amsmath,multicol,rotating,scalefnt}

\newtheorem{thm}[equation]{Theorem}
\numberwithin{equation}{section}
\newtheorem{cor}[equation]{Corollary}

\newtheorem{defin}[equation]{Definition}

\newtheorem{tabl}[equation]{Table}

\begin{document}
\setcounter{MaxMatrixCols}{37}
\def\dig{\begin{picture}(15,10)\put(-9,16){\line(3,-2){27}}\put(7.5,5){\circle*{6}}\end{picture}}
\def\digs{\begin{picture}(15,10)\put(-9,16){\line(3,-2){18}}\put(7.5,5){\circle*{6}}\end{picture}}

\def\mc{\multicolumn}
\def\ss{\smallskip}
\def\bs{\bigskip}
\def\ssum{\sum\limits}
\def\dsum{{\displaystyle{\sum}}}
\def\la{\langle}
\def\ra{\rangle}
\def\on{\operatorname}
\def\gcd{\on{gcd}}
\def\a{\alpha}
\def\bz{{\Bbb Z}}
\def\eps{\epsilon}
\def\br{{\bold R}}
\def\bc{{\bold C}}
\def\bN{{\bold N}}
\def\cm{{\bold m}}
\def\nut{\widetilde{\nu}}
\def\tfrac{\textstyle\frac}
\def\b{\beta}
\def\G{\Gamma}
\def\D{\Delta}
\def\g{\gamma}
\def\zt{{\bold Z}_2}
\def\zth{{\bold Z}_2^\wedge}
\def\bg{{\bold g}}
\def\bof{{\bold f}}
\def\bq{{\bold Q}}
\def\be{{\bold e}}
\def\xb{{\overline x}}
\def\xbar{{\overline x}}
\def\ybar{{\overline y}}
\def\zbar{{\overline z}}
\def\ebar{{\overline \be}}
\def\nbar{{\overline n}}
\def\rbar{{\overline r}}
\def\Ubar{{\overline U}}
\def\et{{\widetilde e}}
\def\rt{{\widetilde R}}
\def\ni{\noindent}
\def\ms{\medskip}
\def\ehat{{\hat e}}
\def\xhat{{\widehat x}}
\def\nbar{{\overline{n}}}
\def\minp{\min\nolimits'}
\def\N{{\Bbb N}}
\def\Z{{\Bbb Z}}
\def\S{{\Bbb S}}
\def\M{{\Bbb M}}
\def\C{{\Bbb C}}
\def\el{\ell}
\def\mo{\on{mod}}
\def\TCB{\overline{\on{TC}}}
\def\lcm{\on{lcm}}
\def\dstyle{\displaystyle}
\def\Remark{\noindent{\it  Remark}}
\title[Topological complexity of lens spaces and $ku$-homology]
{Topological complexity of 2-torsion lens spaces and $ku$-(co)homology}
\author{Donald M. Davis}
\address{Department of Mathematics, Lehigh University\\Bethlehem, PA 18015, USA}
\email{dmd1@lehigh.edu}
\date{February 10, 2015}
\keywords{Topological complexity, lens space, K-theory}
\thanks {2000 {\it Mathematics Subject Classification}:
55M30, 55N15.}

\maketitle
\begin{abstract} We use $ku$-cohomology to determine lower bounds for the topological complexity of mod-$2^e$ lens spaces.
In the process, we give an almost-complete determination of $ku_*(L^\infty(2^e))\otimes_{ku_*}ku_*(L^\infty(2^e))$,
proving a conjecture of Gonz\'alez about the annihilator ideal of the bottom class.
Our proof involves an elaborate row reduction of  presentation
matrices of arbitrary size.
\end{abstract}

\baselineskip=16pt
\section{Main Theorems}\label{sec1} The determination of the topological complexity of topological spaces has been much studied since
its introduction by Farber in \cite{Far}. The (normalized) topological complexity, $\TCB(X)$, of a space $X$ is 1 less than
the smallest number of open subsets of $X\times X$ over which the fibration $PX\to X\times X$, which sends a path $\sigma$ to $(\sigma(0),\sigma(1))$,
has a section. See \cite{Gon2} and \cite{GVW} for an expanded discussion of this concept, especially as it relates to lens spaces.

Let $L^{2n+1}(t)$ denote the standard $(2n+1)$-dimensional $t$-torsion lens space, and let $b(n,e)$, as defined in \cite{GVW},
denote the smallest integer $k$ such that there exists a map
\begin{equation}\label{Lmap}L^{2n+1}(2^e)\times L^{2n+1}(2^e)\to L^{2k+1}(2^e)\end{equation}
which when followed into $L^\infty(2^e)$ is homotopic to a restriction of the $H$-space multiplication of $L^\infty(2^e)=B\Z/2^e$.
In \cite{Gon2}, it is proved that
$$2b(n,e)\le\TCB(L^{2n+1}(2^e))\le 2b(n,e)+1.$$
 Thus
the following theorem yields a lower bound for $\TCB(L^{2n+1}(2^e))$.
Here and throughout $\a(n)$ denotes the number of 1's in the binary expansion of $n$.

\begin{thm}\label{bthm} If $e\ge2$ and $e\le\a(m)<2e$, then
$$b(m+2^{\a(m)-e}-1,e)\ge 2m-2^{\a(m)-e}.$$
\end{thm}

This immediately implies the following result for topological complexity, which might be considered our main result.
\begin{cor} If $e\ge2$ and $e\le\a(m)<2e$, then
$$\TCB(L^{2m+2^{\a(m)-e+1}-1}(2^e))\ge 4m-2^{\a(m)-e+1}.$$
\end{cor}

Other results follow from this and the obvious relation $b(n+1,e)\ge b(n,e)$.
The author believes that this result contains all lower bounds for $b(n,e)$ implied by 2-primary connective complex $K$-theory $ku$.
In \cite{GZ}, a much stronger conjectured lower bound for $b(n,e)$ is given, with the same flavor as our theorem.
Their conjecture depends on conjectures about $BP^*(L^{2n+1}(2^e)\times L^{2n+1}(2^e))$, while our theorem depends on a theorem about
$ku^*(L^{2n+1}(2^e)\times L^{2n+1}(2^e))$.

Our first new result for topological complexity is
$$\TCB(L^{2m+7}(2^{\a(m)-2}))\ge 4m-8\text{ if }\a(m)\ge4.$$

Our theorem is proved by applying $ku^*(-)$ to the map (\ref{Lmap}), obtaining a contradiction under appropriate choice of parameters.
Our main ingredient is the almost-complete determination of $ku^{4n-2d}(L^{2n}(2^e)\times L^{2n}(2^e))$.
It is well-known that $ku^*=\Z_{(2)}[u]$ with $|u|=-2$ and that its $2^e$-series satisfies
$$[2^e](x)=\sum_{i=1}^{2^e}\tbinom{2^e}iu^{i-1}x^i.$$
It is proved in \cite[Proposition 3.1]{Gon1} that
\begin{equation}\label{coh}ku^{\text{ev}}(L^{2n}(2^e)\times L^{2n}(2^e))=ku^*[x,y]/(x^{n+1},y^{n+1},[2^e](x),[2^e](y)),\end{equation}
where $|x|=|y|=2$. One of our main accomplishments is to give a more useful description of $ku^{4n-2d}(L^{2n}(2^e)\times L^{2n}(2^e))$.

On the other hand, $ku$-homology, $ku_*(L_{2^e})$, of
 the infinite-dimensional  lens space $L_{2^e}=L^\infty(2^e)$ is  the $ku_*$-module
generated by classes $z_i$, $i\ge0$, of grading $2i+1$ with relations
$$\sum_{\ell=0}^i\tbinom{2^e}{\ell+1}u^\ell z_{i-\ell},\quad i\ge0.$$ Here $|u|=2$ in $ku_*$.
Also, $ku_*(L_{2^e}\times L_{2^e})$ contains $ku_*(L_{2^e})\otimes_{ku_*}ku_*(L_{2^e})$ as a direct
$ku_*$-summand. We define
$$M_e:=ku_*(\Sigma^{-1}L_{2^e})\otimes_{ku_*}ku_*(\Sigma^{-1}L_{2^e}).$$ It is a  $ku_*$-module on classes $[i,j]:=z_i\otimes z_j$ of grading $2i+2j$, $i,j\ge0$, with relations
\begin{equation}\label{reln}\sum_{\ell=0}^i\tbinom{2^e}{\ell+1}u^\ell[i-\ell,j],\ i,j\ge0,\text{ and }\sum_{\ell=0}^j\tbinom{2^e}{\ell+1}u^\ell[i,j-\ell],\ i,j\ge0.\end{equation}
The desuspending was just for notational convenience. Note that the component of $M_e$ in grading $2d$, which we denote by $G_d$, is isomorphic to $ku^{4n-2d}(L^{2n}(2^e)\times L^{2n}(2^e))$ under the correspondence
$$u^k[i,j]\leftrightarrow u^kx^{n-i}y^{n-j}.$$
 Much of our work goes into an almost-complete description of $M_e$. The result is described in Section \ref{descripsec}.

In \cite[Theorem 2.1]{GVW},  it is proved that
the ideal $$I_e:=(2^e,\ 2^{e-1}u,\ 2^{e-2}u^{3\cdot2-2},\ 2^{e-3}u^{3\cdot2^2-2},\ldots,2^{1}u^{3\cdot2^{e-2}-2},\ u^{3\cdot2^{e-1}-2})$$
annihilates the bottom class $[0,0]$ of $M_e$, and in \cite[Conjecture 2.1]{GVW} it is conjectured that $I_e$ is precisely the annihilator ideal of $[0,0]$ in $M_e$.
One of our main theorems is that this conjecture is true.
\begin{thm}\label{annthm} For $e\ge1$, the annihilator ideal of $[0,0]$ in $M_e$ is precisely $I_e$.\end{thm}
\noindent This is immediate from our description of $M_e$ in Section \ref{descripsec}. See the remark preceding Theorem \ref{genthm}.

After describing $M_e$ in Section \ref{descripsec}, we use this description in Section \ref{TCsec} to prove Theorem \ref{bthm}.
In Section \ref{pfsec}, we prove our result for $M_6$, and in Section \ref{gensec}, we explain how this proof generalizes to arbitrary $M_e$.
Finally, in Section \ref{annsec}, we give a different proof of Theorem \ref{annthm} for $e\le5$, one which is easily checked by a simple computer verification.

The author wishes to thank Jes\'us Gonz\'alez for suggesting this problem, some guidance as to method, and for providing some computer results which were very helpful for finding a general proof.

\section{Description of $M_e$}\label{descripsec}
Our approach to describing $M_e$ is via an associated matrix $P_e$ of polynomials, which we row reduce. The row-reduced form of $P_e$ is quite complicated,
and involves some polynomials which are not completely determined. That is why we call our description ``almost complete.'' In this section, we approach the description of $M_e$ in three steps.

First we give an introduction to our method, define the polynomial matrices $P_e$, and give in Table \ref{init}, without proof, the reduced form of $P_4$, obtained without a computer.
The result for $P_4$ is not used in our general proof, but provides a useful example for comparison. Jes\'us Gonz\'alez obtained an equivalent
result using {\tt Mathematica}.

Next we give in Theorem \ref{longthm} an almost-complete description of the reduced form of $P_6$. This incorporates all aspects of the general reduced $P_e$, but is still describable
in a moderately tractable way.
Finally we give in Theorem \ref{genthm} the general result for $P_e$, which involves a plethora of indices.

Let $G_d$ denote the component of  $M_e$ in grading $2d$. Our ordered set of generators for $G_d$ is
\begin{equation}\label{gens}[0,d],\ldots,[d,0],u[0,d-1],\ldots,u[d-1,0],\ldots,u^d[0,0].\end{equation}
Our final presentation matrix of $G_d$ will be a partitioned matrix
$$\begin{pmatrix}M_{0,0}&M_{0,1}&\ldots&M_{0,d}\\
\vdots&\vdots&\vdots&\vdots\\
M_{d,0}&M_{d,1}&\ldots&M_{d,d}\end{pmatrix},$$
where $M_{i,j}$ is a $(d+1-i)$-by-$(d+1-j)$ Toeplitz matrix.  The columns in a block $M_{i,j}$ correspond to monomials $u^j[-,-]$.

We will use polynomials to represent the submatrices $M_{i,j}$. A polynomial or power series $p(x)=\a_0+\a_1x+\a_2x^2+\cdots$
corresponds to a Toeplitz matrix (of any size) with $(j+k,j)$ entry equal to $\a_k$. Thus the matrix is
$$\begin{pmatrix}a_0&0&0&\\ \a_1&\a_0&0&\ddots\\ \a_2&\a_1&\a_0&\\ \a_3&\a_2&\a_1&\ddots\\ \a_4&\ddots&\ddots&\ddots\\ \vdots&\vdots&\vdots&\vdots
\end{pmatrix}$$
We define $P_e$ to be the polynomial matrix associated to the partitioned presentation matrix of $M_e$ corresponding to the generators (\ref{gens}) and relations (\ref{reln}).
In   (\ref{inl}) we depict $P_6$.

We let
$$p_n(x)=\frac{x^n-1}{x-1}=1+x+\cdots+x^{n-1}.$$
We will display a single upper-triangular matrix of polynomials, whose restriction to the first $d+1$ columns  yields a presentation of $G_d$ for all $d$.
For example, we will see that the first 8 columns for the reduced form of $P_4$  are
$$\begin{pmatrix}16&0&0&4xp_2(x)&0&0&0&2xp_6(x)\\
&8&0&4p_3(x)&0&0&0&2p_7(x)\\
&&8&0&0&0&0&0\\ &&&8&0&0&0&0\\ &&&&4&0&0&0\\ &&&&&4&0&0\\ &&&&&&4&0\\ &&&&&&&4\end{pmatrix}.$$
This implies that a presentation matrix of $G_7$ is as below.
$$\begin{pmatrix}16I_8&0&0&M_{0,3}&0&0&0&M_{0,7}\\
&8I_7&0&M_{1,3}&0&0&0&M_{1,7}\\
&&8I_6&0&0&0&0&0\\
&&&8I_5&0&0&0&0\\ &&&&4I_4&0&0&0\\ &&&&&4I_3&0&0\\ &&&&&&4I_2&0\\ &&&&&&&4I_1\end{pmatrix},$$
where $I_t$ is a $t$-by-$t$ identity matrix, and
$$M_{0,3}=\begin{pmatrix}0&0&0&0&0\\ 4&0&0&0&0\\ 4&4&0&0&0\\ 0&4&4&0&0\\ 0&0&4&4&0\\ 0&0&0&4&4\\ 0&0&0&0&4\\ 0&0&0&0&0\end{pmatrix},\quad
M_{0,7}=\begin{pmatrix}0\\ 2\\ 2\\ 2\\ 2\\ 2\\ 2\\ 0\end{pmatrix},\quad M_{1,3}=\begin{pmatrix}4&0&0&0&0\\ 4&4&0&0&0\\ 4&4&4&0&0\\
0&4&4&4&0\\ 0&0&4&4&4\\ 0&0&0&4&4\\ 0&0&0&0&4\end{pmatrix},\quad M_{1,7}=\begin{pmatrix}2\\ 2\\ 2\\ 2\\ 2\\ 2\\ 2\end{pmatrix}.$$

 The precise reduced form of $P_4$ is as in Table \ref{init}.
We do not offer a proof here, but can prove it by the methods of Section \ref{pfsec}.
 We often write $p_k$ instead of $p_k(x)$.

\bs
\bigskip
\begin{tabl}\label{init}
\begin{center}
{\scalefont{.6}{
$\renewcommand\arraystretch{1.2}\begin{array}{r|ccccccccccccccccc}
&0&1&2&3&4&5&6&7&8&9&10&11&12&13&14&15&16\\
\hline
0&16&0&0&4xp_2&0&0&0&2xp_6&0&0&0&0&0&0&0&xp_{14}&0\\
1&&8&0&4p_3&0&0&0&2p_7&0&0&0&0&0&0&0&p_{15}&0\\
2&&&8&0&0&0&0&0&2x^2p_2(x^2)&2xp_6&0&0&0&0&0&0&x^2p_6(x^2)\\
3&&&&8&0&0&0&0&0&2x^2p_2(x^2)&0&0&0&0&0&0&0\\
4&&&&&4&0&0&0&2p_3(x^2)&2xp_2(x^3)&0&0&0&0&0&0&p_7(x^2)\\
5&&&&&&4&0&0&0&2p_3(x^2)&0&0&0&0&0&0&0\\
6&&&&&&&4&0&0&0&0&0&0&0&0&0&0\\
7&&&&&&&&4&0&0&0&0&0&0&0&0&0\\
8&&&&&&&&&4&0&0&0&0&0&0&0&0\\
9&&&&&&&&&&4&0&0&0&0&0&0&0\\
10&&&&&&&&&&&2&0&0&0&0&0&0\\
11&&&&&&&&&&&&2&0&0&0&0&0\\
12&&&&&&&&&&&&&2&0&0&0&0\\
13&&&&&&&&&&&&&&2&0&0&0\\
14&&&&&&&&&&&&&&&2&0&0\\
15&&&&&&&&&&&&&&&&2&0\\
16&&&&&&&&&&&&&&&&&2
\end{array}$}}
\end{center}
\end{tabl}

\bs
{\scalefont{.53}{
$\renewcommand\arraystretch{1.3}\begin{array}{r|cccccccc}
&17&18&19&20&21&22\\
\hline
0&0&0&0&x^7p_4(x^2)&x^5p_2(x^2)p_4(x^3)&0\\
1&0&0&0&x^6p_8&x^4p_2(x^{12})&0\\
2&xp_6p_3(x^4)&0&x^3p_3p_2(x^2)p_2(x^7)&x^4p_4p_2(x^7)&xp_2p_2(x^{16})+x^8p_2(x^3)&0\\
3&x^2p_6(x^2)&0&0&x^5p_2(x^3)p_2(x^4)&x^4p_2(x^7)p_4&0\\
4&xp_2(x^3)p_3(x^4)&0&x^3p_2(x^2)p_2(x^7)&x^4p_2(x^2)p_2(x^6)&xp_2(x^9)(1+x^2p_3+x^6)&0\\
5&p_7(x^2)&0&0&x^5p_2p_2(x^4)&x^4p_2(x^2)p_2(x^6)&0\\
6&0&x^4p_2(x^4)&0&x^2p_6(x^2)&x^5p_2p_2(x^4)&0\\
7&0&0&x^4p_2(x^4)&0&x^2p_6(x^2)&0\\
8&0&0&0&x^4p_2(x^4)&0&0\\
9&0&0&0&0&x^4p_2(x^4)&0\\
10&0&p_3(x^4)&0&x^2p_2(x^6)&0&0\\
11&0&0&p_3(x^4)&0&x^2p_2(x^6)&0\\
12&0&0&0&p_3(x^4)&0&0\\
13&0&0&0&0&p_3(x^4)&0\\
14&0&0&0&0&0&0\\
15&0&0&0&0&0&0\\
16&0&0&0&0&0&0\\
17&2&0&0&0&0&0\\
18&&2&0&0&0&0\\
19&&&2&0&0&0\\
20&&&&2&0&0\\
21&&&&&2&0\\
22&&&&&&1
\end{array}$}}

\bs

\ni The abelian group that the associated matrix of numbers presents has 276 generators and 276 relations. This associated matrix of numbers is almost, but not quite, in Hermite form.
For example, the polynomial in position $(2,17)$ contains terms such as $2x^5$, and so the associated matrix of numbers will have some 2's sitting far above 2's at the bottom
of the column. For the matrix to be Hermite, all nonzero entries above a 2 at the bottom should be 1's. We could obtain such a polynomial in position $(2,17)$ by subtracting
$(x^5+x^6+x^9+x^{10})$ times row 17 from row 2. We have chosen not to do this here because it will be important to our reduction that the first three nonzero entries in column 17
are $\frac12p_3(x^4)$ times the corresponding entries of column 9.

By restricting to $G_1$, the 8 in position $(1,1)$ shows that $8u[0,0]=0$ in $M_4$. Similarly, by restriction to $G_4$, the 4 in position $(4,4)$ implies that $4u^4[0,0]=0$. We also
obtain $2u^{10}[0,0]=0$ and $u^{22}[0,0]=0$ from the matrix. The Hermite form of the associated matrix of numbers implies that $8[0,0]$, $4u^3[0,0]$, $2u^9[0,0]$, and $u^{21}[0,0]$
are all
nonzero, and this implies Theorem \ref{annthm} when $e=4$.

Next we describe  the reduced  form of $P_6$. We let $P_{i,j}$ denote the entry in  row $i$ and column $j$, where the numbering of each starts with 0.
Throughout the paper, the same notation $P_{i,j}$ will be used for entries in the matrix at any stage of the reduction.
\begin{thm}\label{longthm} The reduced form of the matrix $P_6$ is upper-triangular
with diagonal  entries
$$P_{i,i}=\begin{cases}64&i=0\\32&1\le i\le3\\
16&4\le i\le9\\
8&10\le i\le21\\
4&22\le i\le 45\\
2&46\le i\le93\\1&i=94.\end{cases}$$
Other than these, the nonzero entries are as described below.

\ss
\begin{enumerate}
\item[a.] There are none in columns $0$--$2$, $4$--$6$, $10$--$14$, $22$--$30$,  $46$--$62$, and $94$.
\item[b.] The nonzero entries in columns $3$, $7$--$9$, $15$--$17$, $31$--$33$, and $63$--$65$ are as below.

\ss
{\scalefont{.83}{
$\renewcommand\arraystretch{1.3}\begin{array}{r|c|ccc|ccc}
&3&7&8&9&15&16&17\\
\hline
0&16xp_2&8xp_6&&&4xp_{14}&&\\
1&16p_3&8p_7&&&4p_{15}&&\\
2&&&8x^2p_2(x^2)&8xp_6&&4x^2p_6(x^2)&4xp_6p_3(x^4)\\
3&&&&8x^2p_2(x^2)&&&4x^2p_6(x^2)\\
4&&&8p_3(x^2)&8xp_2(x^3)&&4p_7(x^2)&4xp_2(x^3)p_3(x^4)\\
5&&&&8p_3(x^2)&&&4p_7(x^2)\end{array}$}}

\bigskip
{\scalefont{.83}{
$\renewcommand\arraystretch{1.3}\begin{array}{r|ccc|ccc}
&31&32&33&63&64&65\\
\hline
0&2xp_{30}&&&xp_{62}&&\\
1&2p_{31}&&&p_{63}&&\\
2&&2x^2p_{14}(x^2)&2xp_{6}p_7(x^4)&&x^2p_{30}(x^2)&xp_{6}p_{15}(x^4)\\
3&&&2x^2p_{14}(x^2)&&&x^2p_{30}(x^2)\\
4&&2p_{15}(x^2)&2xp_2(x^3)p_7(x^4)&&p_{31}(x^2)&xp_2(x^3)p_{15}(x^4)\\
5&&&2p_{15}(x^2)&&&p_{31}(x^2)\end{array}$}}

\bigskip
\item[c.] The nonzero entries in columns $18$--$21$, $34$--$37$, and $66$--$69$ are as in Table \ref{T3}. Here $B$ refers to everything in the $18$--$21$ block except the $4p_3(x^4)$-diagonal near the bottom.
The $\bullet$s along a diagonal refer to the entry at the beginning of the diagonal. Each letter $q$ refers
to a polynomial. These polynomials are, for the most part, distinct. The meaning of the diagram is that, except for the diagonal near the bottom,
each entry in the middle portion equals $\frac12p_3(x^8)$ times the corresponding entry in the left portion, and similarly for the right portion, as indicated.
More formally, for $18\le j\le21$ and $i<j-8$,
$$P_{i,j+16}=\tfrac12p_3(x^8)\cdot P_{i,j}\text{ and }P_{i,j+48}=\tfrac14p_7(x^8)\cdot P_{i,j}.$$

\bigskip
\begin{tabl}\label{T3}

\begin{center}
{\scalefont{.83}{
$\renewcommand\arraystretch{1.3}\begin{array}{r|cccc|cccc|cccc|}
&18&19&20&21&34&35&36&37&66&67&68&69\\
\hline
0&0&0&4q&4q&&&&&&&&\\
1&0&0&4q&4q&&&&&&&&\\
2&0&4q&4q&4q&&&&&&&&\\
3&0&0&4q&4q&\multicolumn{3}{c}{B\cdot\frac12p_3(x^8)}&&\multicolumn{3}{c}{B\cdot\frac14p_7(x^8)}&\\
4&0&4q&4q&4q&&&&&&&&\\
5&0&0&4q&4q&&&&&&&&\\
6&4x^4p_2(x^4)&0&4q&4q&&&&&&&&\\
7&0&\dig&0&4q&&&&&&&&\\
8&0&0&\dig&0&&&&&&&&\\
9&0&0&0&\digs&&&&&&&&\\
10&4p_3(x^4)&0&4q&4q&2p_7(x^4)&&&&p_{15}(x^4)&&&\\
11&0&\dig&0&4q&&\dig&&&&\dig&&\\
12&0&0&\dig&0&&&\dig&&&&\dig&\\
13&0&0&0&\digs&&&&\digs&&&&\digs
\end{array}$}}
\end{center}
\end{tabl}

\bigskip
\item[d.] Similarly, the nonzero elements in columns 38 to 45 (other than $P_{i,i}$) are as in Table \ref{T2}. If $C$ denotes all the entries except the
$2p_3(x^8)$-diagonal near the bottom, then columns 70 to 77 are filled exactly with $C\cdot \frac12p_3(x^{16})$ together with a $p_7(x^8)$-diagonal going down from $(22,70)$.

\bigskip
\begin{tabl}\label{T2}

\begin{center}
{\scalefont{.83}{
$$\renewcommand\arraystretch{1.3}\begin{array}{c|cccccccc}
&38&39&40&41&42&43&44&45\\
\hline
0&0&0&2q&2q&2q&2q&2q&2q\\
1&0&0&2q&2q&2q&2q&2q&2q\\
2&0&2q&2q&2q&2q&2q&2q&2q\\
3&0&0&2q&2q&2q&2q&2q&2q\\
4&0&2q&2q&2q&2q&2q&2q&2q\\
5&0&0&2q&2q&2q&2q&2q&2q\\
6&0&0&2q&2q&2q&2q&2q&2q\\
7&0&0&0&2q&2q&2q&2q&2q\\
8&0&0&0&0&2q&2q&2q&2q\\
9&0&0&0&0&0&2q&2q&2q\\
10&0&0&2q&2q&2q&2q&2q&2q\\
11&0&0&0&2q&2q&2q&2q&2q\\
12&0&0&0&0&2q&2q&2q&2q\\
13&0&0&0&0&0&2q&2q&2q\\
14&2x^8p_2(x^8)&0&0&0&2q&0&2q&2q\\
15&0&\dig&0&0&0&2q&0&2q\\
16&0&0&\dig&0&0&0&2q&0\\
17&0&0&0&\dig&0&0&0&2q\\
18&0&0&0&0&\dig&0&0&0\\
19&0&0&0&0&0&\dig&0&0\\
20&0&0&0&0&0&0&\dig&0\\
21&0&0&0&0&0&0&0&\digs\\
22&2p_3(x^8)&0&0&0&2q&0&2q&2q\\
23&0&\dig&0&0&0&2q&0&2q\\
24&0&0&\dig&0&0&0&2q&0\\
25&0&0&0&\dig&0&0&0&2q\\
26&0&0&0&0&\dig&0&0&0\\
27&0&0&0&0&0&\dig&0&0\\
28&0&0&0&0&0&0&\dig&0\\
29&0&0&0&0&0&0&0&\digs
\end{array}$$}}
\end{center}
\end{tabl}

\bigskip
\item[e.] Finally, columns 78 to 93 have a form very similar to Table \ref{T2} with $q$ instead of $2q$
and rows going from $0$ to $61$.
The lower two diagonals are  $x^{16}p_2(x^{16})$ coming down from
$(30,78)$ and $p_3(x^{16})$ coming down from $(46,78)$, and these are the only non-leading
nonzero entries in column ${78}$.
\end{enumerate}
\end{thm}

 Now we state the general theorem, of
which Theorem
\ref{annthm} is an immediate consequence, since the first occurrence of $2^{e-\ell}$ along the diagonal occurs in $(3\cdot2^{\ell-1}-2,3\cdot2^{\ell-1}-2)$.

\begin{thm}\label{genthm} Let $P_{i,j}$ denote the entries in the reduced polynomial matrix for $M_e$. The nonzero entries are
\begin{enumerate}
\item[i.] For $0\le s\le e-1$ and $3\cdot2^s-2\le i<4\cdot2^s-2$ and $0\le t\le e-1-s$,
$$P_{i,i+2^{s+1}(2^t-1)}=2^{e-1-s-t}p_{2^{t+1}-1}(x^{2^s}).$$
\item[ii.] For $0\le s\le e-1$ and $2\cdot2^s-2\le i<3\cdot2^s-2$, $P_{i,i}=2^{e-s}$ and, for $2\le t\le e-s$,
$$P_{i,i+2^{s}(2^t-1)}=2^{e-s-t}x^{2^s}p_{2^t-2}(x^{2^s}).$$
\item[iii.] For $3\le t\le e$ and $2^t+2^{t-2}-2\le j\le 2^t+2^{t-1}-3$, there are possibly nonzero entries
$P_{i,j}=2^{e-t}q_{i,j}$ for $0\le i<j-2^{t-1}$, and also, for $1\le v\le e-t$,
$$P_{i,j+2^t(2^v-1)}=2^{e-t-v}p_{2^{v+1}-1}(x^{2^{t-1}})q_{i,j}.$$
\end{enumerate}
\end{thm}

This generalizes Table \ref{init} and Theorem \ref{longthm}.
Note that some of the entries of type ii are among the entries of type iii. Note also that $p_1(x)=1$, and that in part i for $s=e-1$, we usually just consider the smallest value of
$i$.

\section{Proof of Theorem \ref{bthm}}\label{TCsec}
In this section, we prove Theorem \ref{bthm} by proving
the equivalent statement
\begin{equation}\label{b2thm}\text{if }0\le t<e\text{ and }\a(m)=t+e,\text{ then }
b(m+2^t-1,e)\ge 2m-2^t.
\end{equation}
 The case $t=0$ is elementary (\cite[(1.3)]{GVW}) and is omitted.
We will first prove the following cases of  (\ref{b2thm}) and then will show that all other cases follow by naturality.
\begin{thm} \label{2cases} For $1\le t<e$,
\begin{equation}\label{c1} b(3\cdot2^{t-1}-1+2^{t+1}B,e)\ge2^{t+2}B\text{ if }\a(B)=e+t-1,\end{equation}
and
\begin{equation}\label{c2} b(2^t-1+2^tB,e)\ge(2B-1)2^t\text{ if }\a(B)=e+t.\end{equation}
\end{thm}
These are the cases $m=2^{\a(B)-e}(4B+1)$ and $m=2^{\a(B)-e}B$ of Theorem \ref{bthm} or (\ref{b2thm}).
\begin{proof} We focus on (\ref{c1}), and then discuss the minor changes required for (\ref{c2}). Let $n=3\cdot2^{t-1}-1+2^{t+1}B$ and suppose there is a map
$$L^{2n}(2^e)\times L^{2n}(2^e)\to L^{2^{t+3}B-1}(2^e)$$
as in (\ref{Lmap}). Precompose with the self-map $(1,-1)$ of $L^{2n}(2^e)\times L^{2n}(2^e)$, where $-1$ is homotopic to the Hopf inverse of the identity.
Then, as in \cite{Ast}, we obtain $$(x-y)^{2^{t+2}B}=0\in ku^*(L^{2n}(2^e)\times L^{2n}(2^e)).$$

The result (\ref{c1}) will follow from showing that $$(x-y)^{2^{t+2}B}\ne0\in ku^{2(2n-d)}(L^{2n}(2^e)\times L^{2n}(2^e))$$ with $n=3\cdot2^{t-1}-1+2^{t+1}B$
and $d=3\cdot2^t-2$. This group is isomorphic to the component group $G_d$ for $M_e$ whose presentation matrix was described in Section \ref{descripsec}.
The ordered set of generators is obtained as $x^{n-d}y^{n-d}$ multiplied by
\begin{equation}\label{gns}x^0y^d,\ldots,x^dy^0,ux^1y^d,\ldots,ux^dy^1,\ldots,u^dx^dy^d.\end{equation}
We omit the $x^{n-d}y^{n-d}$ throughout our analysis.

One easily shows that
$$\nu\binom{2^{t+2}B}j\begin{cases}=\a(B)&j=2^{t+1}B\\
>\a(B)&0<|2^{t+1}B-j|<2^{t+1}.\end{cases}$$
Here and throughout $\nu(-)$ denotes the exponent of 2 in an integer.
We wish to show that if $t<e$ and $d=3\cdot2^t-2$, then $2^{e+t-1}x^{d/2}y^{d/2}+2^{e+t}f(x,y)\ne0$ in $G_d$, where $f(x,y)$ is a polynomial of degree $d$ in $x$ and $y$.

In the reduced  matrix for $P_e$ we omit all columns and rows not of the form $3\cdot2^i-3$, $0\le i\le t$. Omitting columns amounts to taking a quotient,
and when a column(generator) is omitted the row(relation) with its leading entry can be omitted, too. The resulting matrix is presented below, where the various polynomials $q$ are mostly distinct.

\medskip
\begin{tabl}\label{qtbl}

\begin{center}
{\scalefont{.75}{
$\renewcommand\arraystretch{1.2}\begin{array}{c|ccccccc}
&0&3&9&3\cdot2^3-3&\cdots&3\cdot2^{t-1}-3&3\cdot2^t-3\\
\hline
0&2^e&2^{e-2}p_2&2^{e-3}q&2^{e-4}q&&2^{e-t}q&2^{e-t-1}q\\
3&&2^{e-1}&2^{e-3}p_2(x^2)&2^{e-4}q&&2^{e-t}q&2^{e-t-1}q\\
9&&&2^{e-2}&2^{e-4}p_2(x^4)&&2^{e-t}q&2^{e-t-1}q\\
3\cdot2^3-3&&&&2^{e-3}&&2^{e-t}q&2^{e-t-1}q\\
\vdots&&&&&\ddots&&\\
3\cdot2^{t-1}-3&&&&&&2^{e-t+1}&2^{e-t-1}p_2(x^{2^{t-1}})\\
3\cdot2^t-3&&&&&&&2^{e-t}
\end{array}$}}
\end{center}
\end{tabl}

\medskip
We temporarily ignore the polynomials $q$ and the polynomial $f(x,y)$. The first few relevant relations in the corresponding numerical matrix are
$x^{d/2}y^{d/2}$ times the following polynomials. We omit writing powers of $u$; they equal the degree of the written polynomial.

\medskip

{\scalefont{.75}{
$\renewcommand\arraystretch{1.2}\begin{array}{ccccccc}
2^e&+&2^{e-2}(xy^2+x^2y)&&&&\\
&&2^{e-1}xy^2&+&2^{e-3}(x^3y^6+x^5y^4)&&\\
&&2^{e-1}x^2y&+&2^{e-3}(x^4y^5+x^6y^3)&&\\
&&&&2^{e-2}x^3y^6&+&2^{e-4}(x^7y^{14}+x^{11}y^{10})\\
&&&&2^{e-2}x^4y^5&+&2^{e-4}(x^8y^{13}+x^{12}y^9)\\
&&&&2^{e-2}x^5y^4&+&2^{e-4}(x^9y^{12}+x^{13}y^8)\\
&&&&2^{e-2}x^6y^3&+&2^{e-4}(x^{10}y^{11}+x^{14}y^7).
\end{array}$}}

\medskip
From these relations, we obtain
\begin{eqnarray}2^{e+t-1}&\sim&-2^{e+t-3}(xy^2+x^2y)\label{first}\\
&\sim&2^{e+t-5}(x^3y^6+x^4y^5+x^5y^4+x^6y^3)\nonumber\\
&\sim&-2^{e+t-7}\sum_{i=7}^{14}x^iy^{21-i}\nonumber\\
&\sim&\cdots\nonumber\\
&\sim&\pm2^{e-t-1}\sum_{i=2^t-1}^{2^{t+1}-2}x^iy^{3\cdot2^t-3-i}\nonumber\\
&=&\pm2^{e-t-1}(x^{3\cdot2^{t-1}-2}y^{3\cdot2^{t-1}-1}+x^{3\cdot2^{t-1}-1}y^{3\cdot2^{t-1}-2})\nonumber\\
&\ne&0,\nonumber\end{eqnarray}
since maximum exponents are $3\cdot2^{t-1}-1$. That the last line is nonzero follows from the reduced form of the matrix $M_e$ of relations.

Now we incorporate the polynomials $q$ in the above matrix. We denote by $\cm_i$ a monomial or sum of monomials of degree $3\cdot2^i-3$, in $x$ and $y$.
At the first step of the above reduction sequence, we would have an additional $\sum_{i=2}^t2^{t+e-i-2}\cm_i$. At the second step, we add
\begin{equation}\label{added}\sum_{i=3}^t2^{t+e-i-3}\cm_i'.\end{equation}
We can incorporate the first monomials for $i\ge3$ into the second, and we replace $2^{t+e-4}\cm_2$ by $\sum_{i\ge3}2^{t+e-i-3}\cm''_i$
and incorporate these into (\ref{added}). The third step adds $\sum_{i=4}^t2^{t+e-i-4}\cm_i'''$. We incorporate (\ref{added}) into this for $i>3$, while the term in (\ref{added}) with
$i=3$ is equivalent to a sum which can also be incorporated. Continuing, we end with
$$\sum_{i=t}^t2^{t+e-i-t}\cm_t^{(t)}=2^{e-t}\cm_t^{(t)}=0,$$
so the $q$'s contribute nothing.

We easily see that incorporating $2^{e+t}f(x,y)$ also contributes nothing, since
$$2^{e+t}\cm\sim2^{e+t-2}\cm_1\sim2^{e+t-4}\cm_2\sim\cdots\sim2^{e+t-2t}\cm_t=0.$$

The proof of (\ref{c2}) is very similar. We want to show $(x-y)^{(2B-1)2^t}\ne0$ in $G_{3\cdot2^t-2}$ if $\a(B)=e+t$ and $1\le t<e$. For $(B-2)2^t<j<(B+1)2^t$, we have
$$\nu\binom{(2B-1)2^t}j\begin{cases}=\a(B)-1&\text{if }j=(B-1)2^t\text{ or }B\cdot2^t\\
>\a(B)-1&\text{other }j.\end{cases}$$
We have factored out $x^{n-d}y^{n-d}$ with $n-d=2^tB-2^{t+1}+1$. Our ordered set of generators is again (\ref{gns}), and our class now, mod higher 2-powers, is
$2^{e+t-1}(x^{2^t-1}y^{2^{t+1}-1}+x^{2^{t+1}-1}y^{2^t-1})$.
Utilizing the relations similarly to (\ref{first}), we end with
\begin{eqnarray*}&&\pm2^{e+t-1}(x^{2^t-1}y^{2^{t+1}-1}+x^{2^{t+1}-1}y^{2^t-1})\sum_{i=2^t-1}^{2^{t+1}-2}x^iy^{3\cdot2^t-3-i}\\
&=&\pm2^{e+t-1}( x^{3\cdot2^{t}-3}y^{3\cdot2^{t}-2}+x^{3\cdot2^{t}-2}y^{3\cdot2^{t}-3})\ne0,\end{eqnarray*}
since $x^{d+1}=0=y^{d+1}$ (after factoring out $x^{n-d}y^{n-d}$).
\end{proof}

\medskip
\begin{proof}[Proof of (\ref{b2thm}).]
The proof is by induction on $t$. If $t=1$, the theorem follows from (\ref{c1}) with $m=4B+1$ if $m\equiv1$ mod 4, and from (\ref{c2}) with $m=2B$
if $m$ is even. If $m\equiv3$ mod 4, then $\a(m+1)\le\a(m)-1=e$, so the result follows from the case $t=0$ for $n=m+1$.

Now we assume that the result has been proved for all $t'<t$. If $m$ is odd, then $\a(m-1)=e+t-1$, so using the induction hypothesis in the middle step,
$$b(m+2^t-1,e)\ge b(m-1+2^{t-1}-1,e)\ge 2(m-1)-2^{t-1}\ge 2m-2^t.$$
If $\nu(m+2^t)=k$ with $1\le k\le t-2$, then, noting that $\nu(m)=k$, too,
\begin{eqnarray*}\a(m-2^k)&=&(t+e)-2^k+\nu(m\cdots(m-2^k+1))\\
&=&t+e-2^k+(2^k-1)=t+e-1.\end{eqnarray*}
Therefore
$$b(m+2^t-1,e)\ge b(m-2^k+2^{t-1}-1,e)\ge 2(m-2^k)-2^{t-1}\ge 2m-2^t.$$

If $\nu(m)\ge t$, let $m= 2^tB$ with $\a(B)=\a(m)=t+e$. By (\ref{c2}), we obtain $b(m+2^t-1,e)\ge 2m-2^t$, as desired.
If  $m=2^{t-1}+2^{t+1}B$ with $\a(B)=t+e-1$, then (\ref{c1}) is exactly the desired result.

Finally, if $m=3\cdot2^{t-1}+2^{t+1}A$ with $\a(A)=t+e-2$, then
\begin{eqnarray*}\a(m+2^{t-1})&=&\a(A+1)\\
&=&\a(A)+1-\nu(A+1)\\
&=&e+v\text{ with }v<t.\end{eqnarray*}
  Thus, by the induction hypothesis,
$$b(m+2^t-1,e)\ge b(m+2^{t-1}+2^v-1,e)\ge2(m+2^{t-1})-2^v\ge2m-2^t.$$

\end{proof}
\section{Proof of Theorem \ref{longthm}}\label{pfsec}
In this section we prove Theorem \ref{longthm}. Because it is a fairly complicated row reduction, we accompany the  proof with
diagrams of the matrix at several stages of the reduction. Although the proof of Theorem \ref{genthm} in Section \ref{gensec} is a complete proof and subsumes the much-longer proof
for $e=6$, we feel that the more explicit example renders the general proof more comprehensible, or perhaps unnecessary.

If $M$ is a Toeplitz matrix corresponding to a polynomial $p(x)$ as described in the preceding section, then the Toeplitz matrix corresponding to the polynomial $(1+\a x+\b
x^2)p(x)$
is obtained from $M$ by adding $\a$ times each row to the one below it and $\b$ times each row to the row  2 below it. This illustrates how row operations
on the matrix of polynomials correspond to row operations on the partitioned matrix of numbers.

Our matrices  now refer to the case $e=6$.
The initial partitioned matrix for $G_d$ could be considered as the matrix of numbers associated to the following matrix of polynomials, which has
$d+1$ columns and $2(d+1)$ rows.
\begin{equation}\renewcommand\arraystretch{1.5} \begin{pmatrix}\label{inl} 64&\tbinom{64}2x&\tbinom{64}3x^2&\tbinom{64}4x^3&\cdots\\
64&\tbinom{64}2&\tbinom{64}3&\tbinom{64}4&\cdots\\
0&64&\tbinom{64}2x&\tbinom{64}3x^2&\cdots\\
0&64&\tbinom{64}2&\tbinom{64}3&\cdots\\
0&0&64&\tbinom{64}2x&\cdots\\
0&0&64&\tbinom{64}2&\cdots\\
&&\vdots&&\end{pmatrix}\end{equation}
The first two row blocks of the associated matrices of numbers have $d+1$ rows of numbers, the next two $d$ rows, etc., while the sizes of the column blocks are
$d+1,d,\ldots$. The first (resp.~second) (resp.~third) row block corresponds to the first (resp.~second) (resp.~first) set of relations in (\ref{reln}) with $i+j=d$ (resp.~$d$)
(resp.~$d-1$).

Note that if the first two rows and the first column of (\ref{inl}) are deleted, we obtain exactly the initial matrix for $G_{d-1}$.
We may assume that the matrix for $G_{d-1}$ has already been reduced, to $Q_{d-1}$. Thus we may obtain the reduced form for $G_d$ by taking  $Q_{d-1}$,
placing a column  of 0's in front of it and the top two rows of (\ref{inl}) above that, and then reducing. By the nature of the matrix (\ref{inl}),
the restriction of the reduced form  $Q_d$ to its first $d$ columns will be  $Q_{d-1}$.

This is an interesting property. Let $Q_d$ denote the reduced form of the polynomial matrix for $G_d$. Remove its last column, put a column of 0's in front, put the top two rows
of (\ref{inl}) above this, and reduce. The result will be the original matrix, $Q_d$.
We will prove that the matrix  described in Theorem \ref{longthm} is correct
by removing its last column (the one with the 1 at the bottom), preceding the matrix by a column of 0's and this by the first two rows of (\ref{inl}), and seeing that
after reducing, we obtain the original matrix $Q_{94}$. Because of the initial shifting, each column is determined by the column which precedes it, together with the reduction steps,
which justifies the method of starting with the
putative answer, shifted.
This seems to be a rather remarkable proof.  However, the reduction is far from being a simple matter.

Now we describe the steps in the reduction.
We begin with the putative answer pushed one unit to the right and two units down, preceded by the first two rows of (\ref{inl}) and a column of 0's.
We often write $R_i$ and $C_j$ for row $i$ and column $j$.

{\bf Step 0}: Subtract $R_0$ from $R_1$, then divide $R_1$ by $(1-x)$, and then subtract $xR_1$ from $R_0$. These rows become
$$\renewcommand\arraystretch{1.5} \begin{pmatrix}64&0&-\binom{64}3x&-\binom{64}4xp_2&-\binom{64}5xp_3&\cdots&-\binom{64}{64}xp_{62}&0&\cdots\\
0&\binom{64}2&\binom{64}3p_2&\binom{64}4p_3&\binom{64}5p_4&\cdots&\binom{64}{64}p_{63}&0&\cdots
\end{pmatrix}$$

Divide $R_1$ by 63, which is the unit part of $\binom{64}2$.
We now have, in $R_0$ and $R_1$,
$$P_{i,j}=\begin{cases}64&i=j=0\\
32&i=j=1\\
0&i+j=1\\
u_j2^{6-\nu(j+1)}xp_{j-1}&i=0,\ 2\le j\le 63\\
u_j'2^{6-\nu(j+1)}p_j&i=1,\ 2\le j\le 63\\
0&0\le i\le 1,\ 64\le j\le94,
\end{cases}$$
where $u_j$ is the odd factor of $-\binom{64}{j+1}$, and $u_j'\equiv u_j$ mod 64.

Our goal is to reduce this matrix so that the first nonzero entry (which we often call the ``leading entry'') in $R_i$ is
\begin{equation}\label{goal}\begin{cases}64\text{ in }C_0&i=0\\
32\text{ in }C_1&i=1\\
16\text{ in }C_4&i=2\\
32\text{ in }C_{i-1}&3\le i\le4\\
8\text{ in }C_{10}&i=5\\
16\text{ in }C_{i-1}&6\le i\le10\\
4\text{ in }C_{22}&i=11\\
8\text{ in }C_{i-1}&12\le i\le22\\
2\text{ in }C_{46}&i=23\\
4\text{ in }C_{i-1}&24\le i\le46\\
1\text{ in }C_{94}&i=47\\
2\text{ in }C_{i-1}&48\le i\le 94.
\end{cases}\end{equation}

\noindent The above entries for $i=0$, 1, 2, 5, 11, 23, and 47 will be the only nonzero entry in their columns.
Then we rearrange rows. For $i=2$, 5, 11, 23, and 47, $R_i$ moves to position $2i$. For other values of $i>2$, $R_i$ moves to position $i-1$.
Then we are finished. The entries $P_{i,i}$ will be as stated in Theorem \ref{longthm}, and the matrix will be upper triangular
with nonzero entries above the diagonal less 2-divisible than the diagonal entry in their column.

Table \ref{starttbl} depicts
the first 22 columns of the matrix at the end of Step 0, except that we omit writing the odd factors in rows 0 and 1.

\ss
\bigskip
\begin{tabl}\label{starttbl}

\begin{center}
{\scalefont{.65}{
$\renewcommand\arraystretch{1.2}\begin{array}{r|ccccccccccccc}
&0&1&2&3&4&5&6&7&8&9&10&11&12\\
\hline
0&64&0&64x&16xp_2&64xp_3&32xp_4&64xp_5&8xp_6&64xp_7&32xp_8&64xp_9&16xp_{10}&64xp_{11}\\
1&&32&64p_2&16p_3&64p_4&32p_5&64p_6&8p_7&64p_8&32p_9&64p_{10}&16p_{11}&64p_{12}\\
2&&64&0&0&16xp_2&0&0&0&8xp_6&0&0&0&0\\
3&&&32&0&16p_3&0&0&0&8p_7&0&0&0&0\\
\hline
4&&&&32&0&0&0&0&0&8x^2p_2(x^2)&8xp_6&0&0\\
5&&&&&32&0&0&0&0&0&8x^2p_2(x^2)&0&0\\
6&&&&&&16&0&0&0&8p_3(x^2)&8xp_2(x^3)&0&0\\
7&&&&&&&16&0&0&0&8p_3(x^2)&0&0\\
\hline
8&&&&&&&&16&0&0&0&0&0\\
9&&&&&&&&&16&0&0&0&0\\
10&&&&&&&&&&16&0&0&0\\
11&&&&&&&&&&&16&0&0\\
12&&&&&&&&&&&&8&0\\
13&&&&&&&&&&&&&8
\end{array}$}}
\end{center}
\end{tabl}

\ss
{\scalefont{.6}{
$\renewcommand\arraystretch{1.3}\begin{array}{r|ccccccccccc}
&13&14&15&16&17&18&19&20&21\\
\hline
0&32xp_{12}&64xp_{13}&4xp_{14}&64xp_{15}&32xp_{16}&64xp_{17}&16xp_{18}&64xp_{19}&32xp_{20}\\
1&32p_{13}&64p_{14}&4p_{15}&64p_{16}&32p_{17}&64p_{18}&16p_{19}&64p_{20}&32p_{21}\\
2&0&0&0&4xp_{14}&0&0&0&0&4q\\
3&0&0&0&4p_{15}&0&0&0&0&4q\\
\hline
4&0&0&0&0&4x^2p_6(x^2)&4xp_3(x^4)p_6&0&4q&4q\\
5&0&0&0&0&0&4x^2p_6(x^2)&0&0&4q\\
6&0&0&0&0&4p_7(x^2)&4xp_2(x^3)p_3(x^4)&0&4q&4q\\
7&0&0&0&0&0&4p_7(x^2)&0&0&4q\\
\hline
8&0&0&0&0&0&0&4x^4p_2(x^4)&0&4q\\
9&0&0&0&0&0&0&0&4x^4p_2(x^4)&0\\
10&0&0&0&0&0&0&0&0&4x^4p_2(x^4)\\
11&0&0&0&0&0&0&0&0&0\\
\hline
12&0&0&0&0&0&0&4p_3(x^4)&0&4q\\
13&0&0&0&0&0&0&0&4p_3(x^4)&0\\
14&8&0&0&0&0&0&0&0&4p_3(x^4)\\
15&&8&0&0&0&0&0&0&0\\
16&&&8&0&0&0&0&0&0\\
17&&&&8&0&0&0&0&0\\
18&&&&&8&0&0&0&0\\
19&&&&&&8&0&0&0\\
20&&&&&&&8&0&0\\
21&&&&&&&&8&0\\
22&&&&&&&&&8
\end{array}$}}

\bigskip


Although it is just a simple shift, it will be useful to have for reference, in Tables \ref{T3p} and \ref{T2p}, the shifted versions of Tables \ref{T3} and \ref{T2}.
These are the relevant portions of the matrix at the outset of the reduction. The shifted version of part b of Theorem \ref{longthm}
can be mostly seen in Table \ref{starttbl}.

\bigskip
\begin{tabl}\label{T3p}

\begin{center}
{\scalefont{.83}{
$\renewcommand\arraystretch{1.3}\begin{array}{r|cccc|cccc|cccc|}
&19&20&21&22&35&36&37&38&67&68&69&70\\
\hline
&{}&&&&&&&&&&&\\
2&0&0&4q&4q&&&&&&&&\\
3&0&0&4q&4q&&&&&&&&\\
4&0&4q&4q&4q&&&&&&&&\\
5&0&0&4q&4q&\multicolumn{3}{c}{B\cdot\frac12p_3(x^8)}&&\multicolumn{3}{c}{B\cdot\frac14p_7(x^8)}&\\
6&0&4q&4q&4q&&&&&&&&\\
7&0&0&4q&4q&&&&&&&&\\
8&4x^4p_2(x^4)&0&4q&4q&&&&&&&&\\
9&0&\dig&0&4q&&&&&&&&\\
10&0&0&\dig&0&&&&&&&&\\
11&0&0&0&\digs&&&&&&&&\\
12&4p_3(x^4)&0&4q&4q&2p_7(x^4)&&&&p_{15}(x^4)&&&\\
13&0&\dig&0&4q&&\dig&&&&\dig&&\\
14&0&0&\dig&0&&&\dig&&&&\dig&\\
15&0&0&0&\digs&&&&\digs&&&&\digs
\end{array}$}}
\end{center}
\end{tabl}

\bigskip
\begin{tabl}\label{T2p}

\begin{center}
{\scalefont{.83}{
$$\renewcommand\arraystretch{1.3}\begin{array}{c|cccccccc}
&39&40&41&42&43&44&45&46\\
\hline
&{}&&&&&&&\\
2&0&0&2q&2q&2q&2q&2q&2q\\
3&0&0&2q&2q&2q&2q&2q&2q\\
4&0&2q&2q&2q&2q&2q&2q&2q\\
5&0&0&2q&2q&2q&2q&2q&2q\\
6&0&2q&2q&2q&2q&2q&2q&2q\\
7&0&0&2q&2q&2q&2q&2q&2q\\
8&0&0&2q&2q&2q&2q&2q&2q\\
9&0&0&0&2q&2q&2q&2q&2q\\
10&0&0&0&0&2q&2q&2q&2q\\
11&0&0&0&0&0&2q&2q&2q\\
12&0&0&2q&2q&2q&2q&2q&2q\\
13&0&0&0&2q&2q&2q&2q&2q\\
14&0&0&0&0&2q&2q&2q&2q\\
15&0&0&0&0&0&2q&2q&2q\\
16&2x^8p_2(x^8)&0&0&0&2q&0&2q&2q\\
17&0&\dig&0&0&0&2q&0&2q\\
18&0&0&\dig&0&0&0&2q&0\\
19&0&0&0&\dig&0&0&0&2q\\
20&0&0&0&0&\dig&0&0&0\\
21&0&0&0&0&0&\dig&0&0\\
22&0&0&0&0&0&0&\dig&0\\
23&0&0&0&0&0&0&0&\digs\\
24&2p_3(x^8)&0&0&0&2q&0&2q&2q\\
25&0&\dig&0&0&0&2q&0&2q\\
26&0&0&\dig&0&0&0&2q&0\\
27&0&0&0&\dig&0&0&0&2q\\
28&0&0&0&0&\dig&0&0&0\\
29&0&0&0&0&0&\dig&0&0\\
30&0&0&0&0&0&0&\dig&0\\
31&0&0&0&0&0&0&0&\digs
\end{array}$$}}
\end{center}
\end{tabl}

At any stage of the reduction, let $\rt_i$ denote $R_i$ with its leading entry changed to 0. The first nonzero entry of $\rt_i$ at the outset occurs in column
\begin{equation}\label{start}\begin{cases}i+5&4\le i\le5\\
i+3&6\le i\le7\\
i+11&8\le i\le11\\
i+7&12\le i\le15\\
i+23&16\le i\le23\\
i+15&24\le i\le31\\
i+47&32\le i\le47\\
i+31&48\le i\le63.\end{cases}\end{equation}
For $i\ge64$, $\rt_i$ has no nonzero elements.

The relationship between the three parts of Table \ref{T3p} and the similar relationship that columns 71 to 78 are mostly $\frac12p_3(x^{16})$ times
Table \ref{T2p} will be very important.  We call it a ``proportionality'' relation. We extend it to also include that in rows 4, 5, and 6
we have $C_{18}/C_{10}=\frac12 p_3(x^4)$, $C_{34}/C_{10}=\frac14p_7(x^4)$, and $C_{66}/C_{10}=\frac18p_{15}(x^4)$, and similarly in row 4, columns
9, 17, 33, and 65.
 When we perform row operations involving these rows,
these relationships continue to hold. Rows 12--15 and 24--31, where the proportionality relationship does not hold, will not be involved
in row operations, since the columns in which their leading entry occurs have all 0's above the leading entry. (Although rows 0 and 1 are initially nonzero in these
columns, clearing out $R_0$ and $R_1$, as in Step 1 below, is a 2-step process, and so $\rt_i$ for $12\le i\le15$ or $24\le i\le31$ will not be combining into
$R_0$ or $R_1$, either.)


In Steps 3, 6, 9, and 12, we will divide rows 2, 5, 11, and 23 by $x$, $x^2$, $x^4$, and $x^8$. It will be important that the entire rows are divisible
by these powers of $x$. We  keep track of bounds for the $x$-divisibility of the unspecified polynomials in Table \ref{T3p} and \ref{T2p} and in columns 79 to 94.
We postpone this analysis until all the reduction steps have been outlined.
 Similarly to the proportionality considerations just discussed,  divisibility bounds
are preserved when we add a multiple of one row to another, in that the $x$-exponent of $P_{i,j}+cP_{i',j}$ is $\ge$ the minimum of that of $P_{i,j}$ and $P_{i',j}$. The rows, 3,
6--7, 12--15, 24--31, and 48--63, where entries not divisible by $x$ occur will not
be used to modify other rows.

Now we begin an attempt to remove most of the binomial coefficients from $R_0$ and $R_1$.

{\bf Step 1}. The goal is to add multiples of lower rows to $R_0$ and $R_1$ to reduce them to

{\scalefont{.83}{
$$\renewcommand\arraystretch{1.3}\begin{array}{c|cccccccccccccccc}
&0&1&2&3&4&5&6&7&\cdots&15&\cdots&31&\cdots&63&\cdots\\
\hline
0&64&0&0&16xp_2&0&0&0&8xp_6&0&4xp_{14}&0&2xp_{30}&0&xp_{62}&0\\
1&0&32&0&16p_3&0&0&0&8p_7&0&4p_{15}&0&2p_{31}&0&p_{63}&0
\end{array}$$}}

\ni with each 0 referring to all intervening columns. However, we will be forced to bring up some additional entries.
We claim that, after Step 1, the nonzero entries $P_{1,j}$, in addition to those in columns $2^t-1$ listed just above,
are combinations of various
$\rt_j$ with $j\ge5$ and $j$ not in $[6,8]\cup[12,16]\cup[24,32]\cup[48,64]$.
Row 0 is similar but has an extra power of $x$, since this is true at the outset. Rows 0 and 1 will thus have the requisite
proportionality and $x$-divisibility relations.

It will be useful to note that since at the outset all entries in $\rt_i$ for $i\ge2$ are a multiple of $\frac12$ times the leading entry at the bottom of their
column, then, using (\ref{start}),  $2\rt_i$ can be killed (reduced to all 0's) by subtracting multiples of lower rows if $i\ge 32$. For example,  nonzero entries of $\rt_{32}$ occur
only in $C_j$ with $j\ge79$. If the entry in $(32,j)$ is a polynomial $q$, then subtracting $qR_{j+1}$ from $2\rt_{32}$ kills the entry in $C_j$ without changing anything else, since
$\rt_{j+1}=0$ for such $j$.

Similarly $4\rt_i$ can be killed in two steps if $i\ge 16$, and $8\rt_i$ can be killed if $i\ge8$.
We can use this observation to kill the entries in $R_0$ and $R_1$ in many columns.

For example, if $32\le j\le46$, then the numerical coefficient in $P_{0,j}$ and $P_{1,j}$ is 0 mod 8, while there is a leading 4 in $(j+1,j)$. Subtracting multiples of $2R_{j+1}$
from $R_0$ and $R_1$ kills the entries in $(0,j)$ and $(1,j)$ while bringing up multiples of $2\rt_{j+1}$. This can be killed by the observation of the previous two paragraphs.
This method works to eliminate the entries in $R_0$ and $R_1$ in columns 12, 14, 16--18, 20--22, 24--30, and 32--62. (Initial entries in $R_0$ and $R_1$ in columns $>63$ were all
0.) Since $-\binom{64}{2^t}\equiv2^{6-t}$ mod $2^{13-2t}$ for $2\le t\le6$, the entries in $R_0$ and $R_1$ in columns 3, 7, 15,  31, and 63 can be changed to their desired
values with pure 2-power coefficients by similar steps.

For $C_{23}$, we subtract even multiples of $R_{24}$ from $R_0$ and $R_1$ to kill the entries. This brings into $R_0$ and $R_1$ multiples of 4 in some columns 39 to 46 and even
entries in some columns $\ge 71$. The latter entries can be cancelled from below, while cancelling multiples of 4 in $C_j$ for $39\le j\le 46$ brings up multiples of $\rt_{j+1}$.  A
very similar argument and similar conclusion works for removal of entries in $(0,19)$ and $(1,19)$.

 Now we consider $C_{11}$. We subtract multiples of $2R_{12}$ to kill the entries in $R_0$ and $R_1$. This brings up multiples of 8 in $C_{19}$,  $C_{21}$, and $C_{22}$,
  4 in  $C_{j}$ for $35\le j\le46$, and  2 in some columns $>64$, the latter of which can be cancelled from below. We kill the earlier elements with multiples of $R_{j+1}$, leaving a
  combination of the various $\rt_{j+1}$

 Column 9 is eliminated similarly, giving multiples of $\rt_{21}$, $\rt_{37}$, and some others, while columns 8,  10, and 13 are, in a sense, easier
 since their binomial coefficients are 4 times the number at the bottom of their column, rather than 2. For example, to kill the entry in $(1,13)$, we first subtract a multiple of
 $4R_{14}$. This contains a $16q$ in $C_{21}$, which is killed by a multiple of $2R_{22}$. This brings up a $2q'$ in $R_{45}$, the killing of which brings up a multiple of
 $\rt_{46}$.

 To kill the entry in $(1,5)$, we subtract a multiple of $2R_6$, which has entries in $C_j$ for many values of $j\ge9$.
 We can cancel each of these by subtracting a multiple of $R_{j+1}$, accounting for the contributions to $R_1$ of multiples of many $\rt_{k}$ with
 $k\not\in[6,8]\cup[12,16]\cup[24,32]\cup[48,64]$.
Killing the entries in $C_4$ and $C_6$ is similar.

Finally, to kill the entry in $(1,2)$, we subtract a multiple of $2R_3$. This brings up  entries in columns 4, 8, 16, 32, 64, and others, the killing of which brings
up combinations of $\rt_5$, $\rt_9$, $\rt_{17}$, and $\rt_{33}$, as allowed.

\ss

\bigskip
{\bf Step 2}. Subtract $2R_1$ from $R_2$ to remove the 64 in $P_{2,1}$. This brings entries into $R_2$ in columns
\begin{equation}\label{cols}j\in\{3,7,10,15,18,20\text{-}22, 31, 34, 36\text{-}38, 40\text{-}46\}\end{equation}
and others with $j\ge63$.
The entry brought into $C_j$ has numerical coefficient equal to $P_{j+1,j}$.
These are then killed by subtracting corresponding multiples of $R_{j+1}$, which brings up into $R_2$
corresponding multiples of $\rt_{j+1}$ for $j$ as in (\ref{cols}). From $\rt_4$, this will place $q=8x^2p_2(x^2)p_3$ in $C_9$, $\frac12p_3(x^4)q$ in $C_{17}$,
$\frac14p_7(x^4)q$ in $C_{33}$, and $\frac18p_{15}(x^4)q$ in $C_{65}$. This extends the proportionality property of columns 9, 17, 33, 65 to include also
row 2.

Now $R_2$ has $16xp_2$ as its leading entry, in column 4.

{\bf Step 3}. Divide $R_2$ by $xp_2$. Dividing by a polynomial $p$ of the form $1+\sum \a_ix^i$, such as $p_2$, is not a problem. If $M$ is a Toeplitz matrix corresponding to a
polynomial $q$, then the Toeplitz matrix corresponding to $q/p$ is obtained from $M$ by performing the row operations corresponding to finitely many of the terms of the power series
$1/p$.
Dividing by $x$ is more worrisome, and is the reason for much of our work. In this step it is not a big problem, but later, when we have to divide by $x^4$ and $x^8$, more care is
required, which will be handled in Theorem \ref{count} after all steps have been described.

We have the important relation \begin{equation}p_{2t}/p_2=p_t(x^2),\label{pp}\end{equation} which implies that the entries in $P_{2,j}$ for $j=8$, 16, 32, and 64 are now $8p_3(x^2)$, $4p_7(x^2)$, $2p_{15}(x^2)$,
and $p_{31}(x^2)$. The relation (\ref{pp}) and its variants will be used frequently without comment.
In $C_9$, we obtain
$$8x\frac{p_2(x^2)p_3}{p_2}=8xp_2(x^3)+16x^3/p_2.$$
We use $R_{10}$ to cancel the second term, at the expense of bringing up multiples of $x^3\rt_{10}$ into $R_2$. This satisfies proportionality properties, which continue to hold.

\ss
{\bf Step 4}. Subtract $p_3R_2$ from $R_3$ to change $P_{3,4}$ to 0. Since
\begin{equation}\label{peq}p_{2k+1}-p_3p_k(x^2)=-x^2p_{k-1}(x^2),\end{equation}
we obtain $-8x^2p_2(x^2)$ in $P_{3,8}$, $-4x^2p_6(x^2)$ in $P_{3,16}$, and similar expressions in $C_{32}$ and $C_{64}$. We can change the minus to a plus by adding a multiple of $R_9$,
$R_{17}$, etc. This brings up multiples of $\rt_9$, $\rt_{17}$, etc., into $R_3$, but these maintain proportionality and $x$-divisibility properties. Note that $x$-divisibility keeps
changing. For example, in Step 3, that of $R_2$ was decreased by 1, and now all that we can say is that the $x$-divisibility
of $R_3$ is at least the minimum of that of $R_2$ and its previous value for $R_3$. But this will be handled later.

For the convenience of the reader, we list here columns 0 through 10 at this stage of the reduction. Some of the specific polynomials
are not very important, and will later just be called $q$.

\ss
{\scalefont{.85}{
$\renewcommand\arraystretch{1.2}\begin{array}{r|ccccccccccc}
&0&1&2&3&4&5&6&7&8&9&10\\
\hline
0&64&0&0&16xp_2&0&0&0&8xp_6&0&0&8x^3p_2p_2(x^3)\\
1&&32&0&16p_3&0&0&0&8p_7&0&0&8x^2p_2(x^2)p_2(x^3)\\
2&&0&0&0&16&0&0&0&8p_3(x^2)&8xp_2(x^3)&8p_3p_3(x^2)\\
3&&&32&0&0&0&0&0&8x^2p_2(x^2)&8xp_6&8p_3(x^4)\\
\hline
4&&&&32&0&0&0&0&0&8x^2p_2(x^2)&8xp_6\\
5&&&&&32&0&0&0&0&0&8x^2p_2(x^2)\\
6&&&&&&16&0&0&0&8p_3(x^2)&8xp_2(x^3)\\
7&&&&&&&16&0&0&0&8p_3(x^2)\\
\hline
8&&&&&&&&16&0&0&0\\
9&&&&&&&&&16&0&0\\
10&&&&&&&&&&16&0\\
11&&&&&&&&&&&16\\
\end{array}$}}

\medskip
\ss
{\bf Step 5}. Subtract $2R_2$ from $R_5$ to remove the leading entry in $R_5$.  If $P_{2,j}=q$ for $j>4$, then adding $qR_{j+1}$ to $R_5$ will cancel the subtracted entry, at the
expense of adding $q\rt_{j+1}$ to $R_5$.
So $R_5$ gets multiples of $\rt_{j+1}$ for many values of $j$ in the intervals $[8,10]$, $[16,22]$, and $[32,46]$.
The rows that we don't want to bring up are 12--15, 24--31, etc., which contain the lower diagonals in Tables \ref{T3p} and \ref{T2p}, where neither proportionality nor
$x$-divisibility
holds.

\ss
Now the leading entry of $R_5$ is $8x^2p_2(x^2)$ in $C_{10}$.

{\bf Step 6}. Divide $R_5$ by $x^2p_2(x^2)$. We need to know that all entries in $R_5$ are divisible by $x^2$. In Theorem \ref{count}, we will show that this is true for
columns 19--22, 35--46, and 67--94. The only other nonzero entries in $R_5$ are those in columns 10, 18, 34, and 66 with which it started. See Table \ref{starttbl}.
The first nonzero entries in $R_5$ after dividing are $8$ in $C_{10}$ and $4p_3(x^4)$ in $C_{18}$.

\ss
{\bf Step 7}. Subtract multiples of $R_5$ from rows 0, 1, 2, 3, 4, 6, and 7 to clear out $C_{10}$ in these rows. Because it had been the case that
$P_{i,18}/P_{i,10}=\frac12p_3(x^4)$
for $0\le i\le 6$, we will now have $P_{i,18}=0$ for $i\in\{0,1,2,3,4,6\}$. Also, by (\ref{peq}), $P_{7,18}=4(p_7(x^2)-p_3(x^2)p_3(x^4))=-4x^4p_2(x^4)$. We can change the minus to a plus by adding
$x^4p_2(x^4)R_{19}$. Similarly, the only nonzero entries in column 34 (resp.~66) (except for $P_{j+1,j}$) are $2p_7(x^4)$ (resp.~$p_{15}(x^4)$) in $R_5$, and $2x^4p_6(x^4)$
(resp.~$x^4p_{14}(x^4)$) in $R_7$.
This illustrates why the proportionality relations are important.
\ss

{\bf Step 8}. Subtract $2R_5$ from $R_{11}$ to remove the leading entry in $R_{11}$.  Similarly to Step 5, if $P_{5,j}=q$ for $j>10$, then adding $qR_{j+1}$ to $R_{11}$ will cancel
the subtracted entry, at the expense of adding $q\rt_{j+1}$ to $R_{11}$.
So $R_{11}$ gets multiples of $\rt_{j+1}$ for many values of $j$ in the intervals $[18,22]$ and $[34,46]$.

\ss
The first 23 columns now are as below.

\ss
\ss
{\scalefont{.65}{
$\renewcommand\arraystretch{1.2}\begin{array}{r|ccccccccccccccc}
&0&1&2&3&4&5&6&7&8&9&10&11&12&13&14\\
\hline
0&64&0&0&16xp_2&0&0&0&8xp_6&0&0&0&0&0&0&0\\
1&&32&0&16p_3&0&0&0&8p_7&0&0&0&0&0&0&0\\
2&&0&0&0&16&0&0&0&8p_3(x^2)&8q_0&0&0&0&0&0\\
3&&&32&0&0&0&0&0&8x^2p_2(x^2)&8q_1&0&0&0&0&0\\
\hline
4&&&&32&0&0&0&0&0&8x^2p_2(x^2)&0&0&0&0&0\\
5&&&&&0&0&0&0&0&0&8&0&0&0&0\\
6&&&&&&16&0&0&0&8p_3(x^2)&0&0&0&0&0\\
7&&&&&&&16&0&0&0&0&0&0&0&0\\
\hline
8&&&&&&&&16&0&0&0&0&0&0&0\\
9&&&&&&&&&16&0&0&0&0&0&0\\
10&&&&&&&&&&16&0&0&0&0&0\\
11&&&&&&&&&&&0&0&0&0&0\\
12&&&&&&&&&&&&8&0&0&0\\
13&&&&&&&&&&&&&8&0&0\\
14&&&&&&&&&&&&&&8&0\\
15&&&&&&&&&&&&&&&8
\end{array}$}}

\ss
{\scalefont{.6}{
$\renewcommand\arraystretch{1.3}\begin{array}{r|cccccccccc}
&15&16&17&18&19&20&21&22\\
\hline
0&4xp_{14}&0&0&0&0&4q&4q&4q\\
1&4p_{15}&0&0&0&0&4q&4q&4q\\
2&0&4p_{7}(x^2)&8q_0p_3(x^4)&0&4q&4q&4q&4q\\
3&0&4x^2p_{6}(x^2)&8q_1p_3(x^4)&0&4q&4q&4q&4q\\
\hline
4&0&0&4x^2p_6(x^2)&0&0&4q&4q&4q\\
5&0&0&0&4p_3(x^4)&0&4q&4q&4q\\
6&0&0&4p_7(x^2)&0&0&4q&4q&4q\\
7&0&0&0&4x^4p_2(x^4)&0&4q&4q&4q\\
\hline
8&0&0&0&0&4x^4p_2(x^4)&0&4q&4q\\
9&0&0&0&0&0&4x^4p_2(x^4)&0&4q\\
10&0&0&0&0&0&0&4x^4p_2(x^4)&0\\
11&0&0&0&0&0&0&0&4x^4p_2(x^4)\\
\hline
12&0&0&0&0&4p_3(x^4)&0&4q&4q\\
13&0&0&0&0&0&4p_3(x^4)&0&4q\\
14&0&0&0&0&0&0&4p_3(x^4)&0\\
15&0&0&0&0&0&0&0&4p_3(x^4)\\
16&8&0&0&0&0&0&0&0\\
17&&8&0&0&0&0&0&0\\
18&&&8&0&0&0&0&0\\
19&&&&8&0&0&0&0\\
20&&&&&8&0&0&0\\
21&&&&&&8&0&0\\
22&&&&&&&8&0\\
23&&&&&&&&8
\end{array}$}}

\ss
In addition, we have, at this stage of the reduction:

\begin{enumerate}
\item[a.] 4 in $P_{j+1,j}$ for $23\le j\le46$, and 2 in $P_{j+1,j}$ for $47\le j\le94$. Other than that:
\item[b.] 0 in columns 23 to 30 and 47 to 62.
\item[c.] A pattern resembling that of columns 15 to 18  in columns 31 to 34 and 63 to 66.
\item[d.] Columns 35 to 38 (resp.~67 to 70) are $\frac12p_3(x^8)$ (resp.~$\frac14p_7(x^8)$) times columns 19 to 22, except that corresponding to the $4p_3(x^4)$ in rows 12 to 15
we have $2p_7(x^4)$ (resp.~$p_{15}(x^4)$).
\item[e.] Columns 39 to 46 resemble Table \ref{T2p}.  Columns 71 to 78 are $\frac12p_3(x^{16})$ times these,
except for the diagonal near the bottom, which is $p_7(x^8)$.
\item[f.] Columns 79 to 94 have a form similar to that of columns 39 to 46.
\item[g.] The $x$-divisibility in columns 19--22, 39--46, and 79--94  will be described in Theorem \ref{count} and its proof.
\end{enumerate}

\ss

{\bf Step 9}. Divide $R_{11}$ by $x^4p_2(x^4)$. We will show in Theorem \ref{count} that all entries in $R_{11}$ are divisible by $x^4$. The leading entry in row 11 is now a 4
in $C_{22}$.

\ss
{\bf Step 10}. Subtract multiples of $R_{11}$ from rows 0 to 10 and 12 to 15 to clear out their entries in $C_{22}$.  Similarly to Step 7, we now have that  $P_{i,38}=0$ except for
$P_{11,38}=2p_3(x^8)$, $P_{15,38}=2x^8p_2(x^8)$,
and $P_{39,38}=4$, with a similar situation in $C_{70}$. In particular, $P_{15,70}=x^8p_6(x^8)=\frac12p_3(x^{16})P_{15,38}$.

\ss
{\bf Step 11}. Subtract $2R_{11}$ from $R_{23}$, and, similarly to Steps 5 and 8, kill entries subtracted from $P_{23,j}$ for $j>22$ by adding multiples of $R_{j+1}$,
thus bringing up these multiples of $\rt_{j+1}$. The smallest such $j$ is 38, due to the entry in $(11,38)$ described in the previous step.

\ss
{\bf Step 12}. Now the leading entry of $R_{23}$ is $2x^8p_2(x^8)$ in $C_{46}$. (This can be seen using (\ref{start}) and that there have been no other changes to
$R_{23}$ in columns less than 62.) Divide $R_{23}$ by $x^8p_2(x^8)$. We will show later that all entries in $R_{23}$ are divisible by $x^8$ at this stage.

\ss
{\bf Step 13}. Subtract multiples of $R_{23}$ from rows 0 to 22 and 24 to 31 to make their entries in $C_{46}$ equal to 0. Similarly to Step 10, this will cause
$P_{i,78}=0$ except for $P_{23,78}=p_3(x^{16})$, $P_{31,78}=x^{16}p_2(x^{16})$,
and $P_{79,78}=2$.

\ss
{\bf Step 14}. Subtract $2R_{23}$ from $R_{47}$. This will add multiples of 2 to $R_{47}$ in some columns $j\ge78$. These can be removed, without any other effect, by subtracting a
multiple of $R_{j+1}$. Now $R_{47}$ has leading entry $x^{16}p_2(x^{16})$ in $C_{94}$. Divide $R_{47}$ by $x^{16}p_2(x^{16})$, and then subtract multiples of $R_{47}$ from the others
to clear out $C_{94}$.

\ss
{\bf Step 15}. We are now in the situation described in the paragraph containing (\ref{goal}). Rearrange rows as specified there, and we are done.

\bigskip
It remains to show that Steps 3, 6, 9, and 12 above could actually be carried out, by showing that there was sufficient divisibility by $x$.
 This will follow from Theorem \ref{count}.

\begin{defin} Let $\D(0)=3$ and $\D(1)=2$. For $i\ge2$, let $b(i)$ denote the largest integer $\le i$ of the form $2^t-1$ or $3\cdot2^t-1$, and let $\D(i)=i-b(i)$.
\end{defin}

For example, the values of $\D(i)$ for $2\le i\le17$ are as in the following table.

$\begin{array}{c|cccccccccccccccc}
i&2&3&4&5&6&7&8&9&10&11&12&13&14&15&16&17\\
\hline
\D(i)&0&0&1&0&1&0&1&2&3&0&1&2&3&0&1&2
\end{array}$

\ss
\begin{thm} \label{count}Let $\nu(i,j)$ denote the exponent of $x$ in $P_{i,j}$ at any stage of the reduction from the end of Step 1 to the end of Step 14. Then
\begin{itemize}
\item If $19\le j\le22$ and $0\le i\le j-8$, then $\nu(i,j)\ge22-j+\D(i)$.
\item If $39\le j\le 46$ and $0\le i\le j-16$, then $\nu(i,j)\ge46-j+\D(i)$.
\item If $79\le j\le 94$ and $0\le i\le j-32$, then $\nu(i,j)\ge94-j+\D(i)$.
\end{itemize}
\end{thm}

Since this applies to any stage of the reduction, it says that all $x$-exponents in these columns are nonnegative at the end of Steps 3, 6, 9, and 12,
which means that there was enough $x$-divisibility to perform the step. The divisibility of other columns in rows 2, 5, 11, and 23 at Steps 3, 6, 9, and 12 is easily
checked, mostly following from proportionality.

\begin{proof} We give the proof for $79\le j\le94$. The proof for the smaller ranges is basically the same.
The proof is by induction on $j$. By Theorem \ref{longthm}(e) shifted, at the outset $\nu(32,79)=16$, while $\nu(i,79)=\infty$ for $i\ne32$ and $i\le47$.
If $j\ge80$, we assume the result is known for $j-1$.
With the rearranging and shifting, we start with, for $i\ge2$, $$\nu(i,j)=\begin{cases}\nu_E(i-1,j-1)&i\not\in\{2,3, 6, 12, 24, 48\}\\
\nu_E(\frac12i-1,j-1) &i\in\{6,12,24,48\}\\
\nu_E(i-2,j-1)&i\in\{2,3\},\end{cases}$$
where $\nu_E(-,-)$ refers to the value of $\nu$ at the end of Step 14. By the induction hypothesis, this is
$$\ge\begin{cases}94-j+1+\D(i-1)&i\not\in\{2,3,6,12,24,48\}\\ 94-j+1+\D(\frac12i-1)=94-j+1+\D(i-1)&i\in\{6,12,24,48\}\\
94-j+1+5-i&i\in\{2,3\}.\end{cases}$$
Let $\mu(i,j)$ denote a lower bound for $\nu(i,j)-(94-j)$. At the outset, we have, for all $i\ge4$ and $j\ge79$,
$$\mu(i,j)\ge \D(i-1)+1,$$
while $\mu(2,j)\ge 4$ and $\mu(3,j)\ge3$.

We will go through the steps of the reduction and see how $\mu$ changes.  We can dispense with $j$ as part of the notation.
We will now call it $\mu(i)$. To emphasize that $\mu$ is changing, we will let $\mu_k$ denote the value of $\mu$ after Step $k$.
We have $\mu_0(i)\ge\D(i-1)+1$ for $i\ge4$, $\mu_0(2)\ge4$ and $\mu_0(3)\ge3$. Although it is possible that actual divisibility could
increase after a step (by having terms of smallest exponent cancel), our lower bounds, being just bounds, cannot see this.
Thus we always have $\mu_{k+1}(i)\le\mu_k(i)$, so we wish to prove that $\mu_{14}(i)\ge \D(i)$.

Step 1 sets \begin{eqnarray}\label{mu1}\mu_1(1)&\ge&\min(\mu_0(5),\mu_0(9),\mu_0(17),\mu_0(33),\mu_0(10),\\
&&\mu_0(18),\mu_0(34),\mu_0(20),\mu_0(36),\mu_0(40))=2\nonumber\end{eqnarray}
and $\mu_1(0)=\mu_1(1)+1\ge3$. Of course, $\mu_1(i)=\mu_0(i)$ for $i>1$, since Step 1 is only changing $R_0$ and $R_1$. In asserting (\ref{mu1}), it is relevant that
the various $\rt_i$ which affect $R_1$ do not include $i=2^t$ or $3\cdot2^t$, since those are the only $i$ for which $\mu_0(i)=1$.

Step 2 sets
$$\mu_2(2)\ge\min(\mu_1(2),\mu_1(4),\mu_1(8),\mu_1(16),\mu_1(32))\ge1.$$
Other rows that affect $R_2$ would contribute exponents at least this large.
Step 3 subtracts 1 from $\mu_2(2)$, so now $\mu_3(2)\ge0$. Step 4 sets
$$\mu_4(3)\ge\min(\mu_3(3),\mu_3(2))\ge0.$$

We have $\mu_4(5)=\mu_0(5)\ge2$. Step 5 does not change this estimate, i.e., $\mu_5(5)\ge2$, because at Step 5, $R_5$ is not affected by any of the rows, $i=2^t$ with $t\ge1$ or
$i=3\cdot2^t$
with $t\ge0$, for which $\mu_4(i)<2$. This is due to the fact that, for these values of $i$, $\rt_k$ is 0 in $C_{i-1}$ throughout the reduction for all $k\ge2$.
Step 6 subtracts 2 from $\mu(5)$, so now $\mu_6(5)\ge0$.

For Step 7, we need to know the $x$-exponents of the entries in $C_{10}$  at this stage of the reduction. These exponents in row $i$ will be 3, 2, 0, 0, 1, 1, 0 for
$i=0$, 1, 2, 3, 4, 6, and 7. These can be seen in the table at the end of Step 4, or by noting that the entries in rows  4, 6, and 7 will be unchanged from their values in Table \ref{starttbl},
while
$R_1$ got $x^2$ from $\rt_5$ at Step 1, $R_2$ got $x$ from $\rt_4$ at Step 2, then changed to $x^0$ at Step 3, while $R_3$ then got $x^0$ at Step 4. For these values of $i$, we
obtain that $\mu_7(i)$ is $\ge$ the minimum of $\mu_6(i)$ and the exponent listed above. It turns out that the only change is $\mu_7(7)\ge0$. Our bounds now for
$i$ from 0 to 7 are 3, 2, 0, 0, 1, 0, 1, 0.

Since $\mu_7(11)\ge4$ and $\mu_7(i)\ge4$ for $19\le i\le23$ and $35\le i\le47$, we obtain $\mu_8(11)\ge4$, and then $\mu_9(11)\ge0$. For Step 10, we need to know exponent bounds in
$C_{22}$ at this stage of the reduction, because it is these multiples of $R_{11}$ that are being subtracted from the row in question. For $i<15$, they will be the same as the
$\mu$-values that we are computing here, because the same steps apply. However, we have $\mu_{10}(15)=0$
due to the $4p_3(x^4)$-entry in $P_{15,22}$. Our exponent bounds $\mu_{10}(i)$ now for $i$ from 0 to 15 are 3, 2, 0, 0, 1, 0, 1, 0, 1, 2, 3, 0, 1, 2, 3, 0.

Since $\mu_{10}(23)\ge8$ and $\mu_{10}(i)\ge8$ for $39\le i\le47$, we obtain $\mu_{11}(23)\ge8$, and then $\mu_{12}(23)\ge0$. For Step 13, we need to know exponent bounds in
$C_{46}$ at this stage of the reduction, because it is these multiples of $R_{23}$ that are being subtracted from the row in question. For $i<31$, they will be the same as the
$\mu$-values that we are computing here, because the same steps apply. However, we have $\mu_{13}(31)=0$
due to the $2p_3(x^8)$-entry in $P_{31,46}$.

In Step 14, we obtain $\mu_{14}(47)=0$, with no other changes to $\mu$. Our final values for $\mu_{14}(i)$ are 0 for $i=2$, 3, 5, 7, 11, 15, 23, 31, and 47, and increasing in
increments
of 1 from one of these to the next. This equals $\D(i)$, as claimed.
\end{proof}

\bigskip

\section{Proof of Theorem \ref{genthm}}\label{gensec}
 In this section,  we prove Theorem \ref{genthm} by defining a sequence of matrices $N_0,\ldots,N_{e-1}$ at various stages of the reduction, and then show that $N_s$ reduces to $N_{s+1}$. After  its rows are
rearranged,
$N_{e-1}$ will become the matrix described in Theorem \ref{genthm}. We explain in Theorem \ref{steps} how $N_0$ is obtained from $N_{e-1}$.
Comparing with the case $e=6$, $N_0$ through $N_4$ are the matrix after Steps 1, 4, 7, 10, and 13,  respectively, while $N_5$ is the matrix at the end of Step 14
except that $P_{95,94}$ has not yet been cleared out.

In the following, $\lg(-)$ denotes $[\log_2(-)]$, and $\delta_{i,j}$ is the usual Kronecker symbol. We continue to suppress $e$ from the notation.
\begin{defin}\label{Ndef}
For $0\le s\le e-1$, $N_s$ is a matrix with rows numbered from $0$ to $3\cdot2^{e-1}-1$, and columns from $0$ to $3\cdot2^{e-1}-2$ satisfying
\begin{enumerate}
\item[a.]
Its leading entries are
\begin{itemize}
\item $2^e$ in $(0,0)$ and $2^{e-1}$ in $(1,1)$;
\item for $0\le k\le e-1$, $2^{e-k}$ in $(i,i-1)$ for
$$3\cdot2^{k-1}\le i\le3\cdot2^k-\begin{cases}2&k\le s-1\\
1&k\ge s;\end{cases}$$
\item for $1\le\ell\le s$, $2^{e-\ell-1}$ in $(3\cdot2^{\ell-1}-1,3\cdot2^\ell-2)$.
\end{itemize}

\item[b.] For  $2\le\ell+2\le t\le e$, it has
\begin{itemize}
\item $2^{e-t}x^{2^\ell}p_{2^{t-\ell}-2}(x^{2^\ell})$ in $(2^{\ell+1}+m-1-\delta_{\ell+m,0},2^t+2^\ell-2+m)$ for
$$\begin{cases}0\le m\le 2^{\ell}-1&\ell<s\\
0\le m\le2^{\ell}+0&\ell=s\\
1\le m\le2^{\ell}+0&\ell>s;\end{cases}$$
\item $2^{e-t}p_{2^{t-\ell}-1}(x^{2^\ell})$ in $(3\cdot2^\ell+m-1,2^t+2^\ell-2+m)$ for
$$1\le m\le 2^{\ell}+\begin{cases}-1&\ell<s\\ 0&\ell\ge s,\end{cases}$$
and in $(\lceil 3\cdot2^{\ell-1}\rceil-1,2^t+2^\ell-2)$ if $\ell\le s$.
\end{itemize}

\item[c.] Except for the leading entries described in (a),
\begin{itemize}
\item all entries in $C_j$ are $0$ for $3\cdot2^k-1\le j\le 4\cdot2^k-2$, $k\ge0$, as are those
in $C_{3\cdot2^k-2}$ if $k<s$, while the only additional nonzero entry in $C_{3\cdot2^s-2}$ is $2^{e-s-1}$ in row $3\cdot2^{s-1}-1$;
\item if $t\ge2$ and $j=2^t+d$ with $-1\le d\le2^{t-1}-2$, then $P_{i,j}=0$ if $i\ge d+2^{\lg(d+1.5)+1}+2$;
\item if $k\ge0$ and $i=3\cdot2^k-1$, then $P_{i,j}=0$ for $i\le j<2i$.
\end{itemize}

\item[d.] For $3\le t<u\le e$, $2^{t-2}-1\le d\le 2^{t-1}-2$, and $i\le d+2^{t-1}$,
$$P_{i,2^u+d}=\tfrac1{2^{u-t}}p_{2^{u-t+1}-1}(x^{2^{t-1}})P_{i,2^t+d}.$$
This is also true for $d=2^{t-2}-2$ if $s\ge t-2$, except in row $3\cdot2^{t-3}-1$.
\item[e.] If $2\le t\le e-1$, $3\cdot2^t-2^{t-1}-1\le j\le 3\cdot2^t-2$, and $i\le j-2^t$, then $P_{i,j}$ is divisible by $x^\nu$ with $\nu= 3\cdot2^t-2-j+\eta_s(i)$, where
$$\eta_0(i)=\begin{cases}3-i&0\le i\le1\\6-i&2\le i\le 3\\ i-c(i)+1&i\ge4,\end{cases}$$
with $c(i)$  the largest integer $\le i$ of the form $2^v$ or $3\cdot2^v$, and
$$\eta_s(i)=\begin{cases}0&\text{if }i+1=3\cdot2^v\text{ or }4\cdot2^v\text{ for }0\le v<s\\
\eta_0(i)&\text{otherwise.}\end{cases}$$
\end{enumerate}
\end{defin}

 Theorem \ref{genthm} is an immediate consequence of the following result, together with the discussion preceding Step 0 of Section \ref{pfsec}.
\begin{thm} \label{steps} Let $N_s$ denote the matrices of Definition \ref{Ndef}.
\begin{enumerate}
\item[1.] After subtracting $2R_{3\cdot2^{e-2}-1}$ from $R_{3\cdot2^{e-1}-1}$ and then rearranging rows, $N_{e-1}$ satisfies the properties of Theorem \ref{genthm}. Call this
rearranged  matrix $Q$. The rearranging is that for $i=3\cdot2^t-1$ with $0\le t\le e-2$, $R_i$ moves to position $2i$, while for other values of $i>2$, $R_i$ moves to
position $i-1$.
\item[2.] Delete the last column of $Q$, precede this by a column of $0$'s, and precede this by the following two rows.

{\scalefont{.83}{
$$\renewcommand\arraystretch{1.3}\begin{array}{c|ccccccccc}
&0&1&2&3&&2^e-1&2^e&&3\cdot 2^{e-1}-2\\
\hline
0&2^e&\binom{2^e}2x&\binom{2^e}3x^2&\binom{2^e}4x^3&\cdots&x^{2^e-1}&0&\ldots&0\\
1&2^e&\binom{2^e}2&\binom{2^e}3&\binom{2^e}4&\cdots&1&0&\ldots&0
\end{array}$$}}

\noindent Then perform the $2^e$-analogues of Steps 0 and 1 of Section \ref{pfsec}. The result is the matrix $N_0$.
\item[3.] For $0\le s\le e-2$, the matrix $N_s$ reduces to $N_{s +1}$.
\end{enumerate}
\end{thm}
\begin{proof}
Part 1 is straightforward but tedious and mostly omitted.  As an example of the comparison, the final case of the second $\bullet$ of Definition
\ref{Ndef}(b), after rearranging and changing $t$ to $T$, says
$$P_{3\cdot2^\ell-2,2^T+2^\ell-2}=2^{e-T}p_{2^{T-\ell}-1}(x^{2^\ell}).$$
With $\ell=s$ and $T=s+t+1$, this becomes the case $i=3\cdot2^s-2$ of Theorem \ref{genthm}(i).

Next we address Part 2.
 After shifting and performing Step 0 of Section \ref{pfsec}, we will have the $2^e$ analogue of Table \ref{starttbl}, in which we recall that
odd factors were not written. It is easy but tedious to verify that everything except rows 0 and 1 will be as stated for $N_0$.
For example, the first $\bullet$ of Definition \ref{Ndef}(b) with its $s=0$, and $t$ replaced by $T$ becomes
$$P_{2^{\ell+1}+m-1,2^T+2^\ell-2+m}=2^{e-T}x^{2^\ell}p_{2^{T-\ell}-2}(x^{2^\ell})\text{ for }1\le m\le 2^\ell$$
for $\ell>0$.
With $\ell=s$ and $t=T-s$, this matches with
part ii of Theorem \ref{genthm} shifted 2 down and 1 to the right.

Part e of Definition \ref{Ndef} for Part 2 is somewhat delicate. We had $\eta_{e-1}(i)=0$ for $i=2$, 3, 5, 7, 11, 15,$\ldots$, i.e. $i=2^t-1$ or $3\cdot2^t-1$,
with $\eta_{e-1}$ increasing by 1's between these values of $i$. The rearranging done in Part 1 puts these 0's in $i=4$, 2, 10, 6, 22, 14,$\ldots$, i.e. $i=2^t-2$ or $3\cdot2^t-2$,
with $\eta$ again increasing by 1's between these values of $i$. Shifting these down by 2, as is done in Part 2, puts the 0's in $2^t$ and $3\cdot2^t$, starting with $i=4$, but we add 1 to the $\eta$ values
because of the shift of columns. For example, column 21 had $\nu\ge1+\eta$, but this now applies to column 22, where it is interpreted as $0+(\eta+1)$. The values of $\eta_0(2)$ and
$\eta_0(3)$ are 1 greater than $\eta_{e-1}(0)$ and $\eta_{e-1}(1)$, respectively. These values are all as claimed of $\eta_0(i)$ for $i\ge2$.

We kill the terms in $R_0$ and $R_1$ except
for those in columns of the form $2^t-1$ by the method of Step 1 of Section \ref{pfsec}. For example, if $j$ is of the form $3\cdot2^t-1$ or $5\cdot2^t-1$, $t\ge0$, then the
2-exponent in $R_0$ and $R_1$ is 1 greater than that in $R_{j+1}$, which is a leading entry.
We subtract multiples of $2R_{j+1}$ to kill the terms. This brings up multiples of $2\rt_{j+1}$. If this is nonzero in $C_k$, the term brought up can be killed by subtracting a
multiple
of $R_{k+1}$. This brings up multiples of $\rt_{k+1}$. Because columns 11--15, 23--31, etc., i.e. those $j$ satisfying $3\cdot2^t-1\le j\le4\cdot2^t-1$, are 0,
we will not bring up $\rt_i$ for $i$ from 12--16, 24--32, etc., and these are the only rows which contain entries which do not satisfy the proportionality and $x$-divisibility
conditions stated in d and e of Definition \ref{Ndef}, and the only rows that will have $\eta(i)<2$. Thus we will obtain $\eta_0(1)\ge2$,  and $\eta_0(0)\ge3$ since
$R_0$ has an extra factor of $x$ as compared to $R_1$.

Similar reasoning applies to columns $j$ not of the form $3\cdot2^t-1$ or $5\cdot2^t-1$. If also $j\ne 2^t-1$, then the 2-exponent in $R_0$ and $R_1$ will exceed that in
$R_{j+1}$ by more than 1. We can use an even multiple at one of the two steps of the previous paragraph, or can break it up into more steps, which will
make the rows eventually brought up have larger values of $i$, but, either way, we will not be bringing up the bad rows such as 12--16, etc., and so all the properties
will be transferred to $R_0$ and $R_1$. Changing the terms $-\binom{2^e}{2^t}$ in $C_{2^t-1}$ to $2^{e-t}$ is accomplished similarly, using that these differ by a multiple of
$2^{e-t+2}$,
while the entry in $(2^t,2^t-1)$ has 2-exponent $e-t+1$.

There are three steps to the reduction in Part 3, analogous to Steps 5, 6, and 7 in Section \ref{pfsec}. Note that the only nonzero entries of $N_s$ in $C_{3\cdot2^s-2}$ are
$2^{e-s-1}$ in $R_{3\cdot2^{s-1}-1}$, and $2^{e-s}$ in $R_{3\cdot2^s-1}$, and the second nonzero entry in $R_{3\cdot2^s-1}$ is $2^{e-s-2}x^{2^s}p_2(x^{2^s})$ in
$C_{3\cdot2^{s+1}-2}$.
The first step is to subtract $2R_{3\cdot2^{s-1}-1}$ from $R_{3\cdot2^s-1}$. If $\rt_{3\cdot2^{s-1}-1}$ has $q\ne0$ in $C_j$, then the $-2q$ brought into $R_{3\cdot2^s-1}$ can be
killed by adding $qR_{j+1}$. The net effect is to remove the leading entry of $R_{3\cdot2^s-1}$, making the $2^{e-s-2}x^{2^s}p_2(x^{2^s})$ in $C_{3\cdot2^{s+1}-2}$ its new leading
entry,
and to bring into this row various $q\rt_{j+1}$ for which $P_{3\cdot2^{s-1}-1,j}\ne0$. By (c), such $j$ must satisfy $j>2^{s+2}+1$, and then nonzero entries in $\rt_j$ only occur in
columns $>2^{s+3}+1$. This extends the first $\bullet$ in (c) to include also $k=s$, which is needed for $N_{s+1}$.

We must also consider the effect of these changes on $\eta(3\cdot2^s-1)$. We had $\eta_s(3\cdot2^s-1)=\eta_0(3\cdot2^s-1)=2^s$. It follows from (c)  that none of the $j$'s appearing
above
can satisfy $3\cdot2^t-1\le j\le 3\cdot2^t+2^s-3$ or $4\cdot2^{t}-1\le j\le4\cdot2^{t}+2^s-3$, $t\ge s$, which are the only values having $\eta_s(j+1)<2^s$. Thus $\eta(3\cdot2^s-1)$
does not change at this step.

The second step divides $R_{3\cdot2^s-1}$ by $x^{2^s}p_2(x^{2^s})$. This can be done because $\eta_s(3\cdot2^s-1)\ge2^s$.  The dividing changes $\eta(3\cdot2^s-1)$ to 0, which is consistent
with the claim for $\eta_{s+1}(3\cdot2^s-1)$. This step changes $P_{3\cdot2^s-1,2^u+2^{s+1}-2}$ from $2^{e-u}x^{2^s}p_{2^{u-s}-2}(x^{2^s})$ to $2^{e-u}p_{2^{u-s-1}-1}(x^{2^{s+1}})$
for $u\ge s+2$.  It removes the entry in the first $\bullet$ of (b) with $\ell=s$, $m=2^s$, $t=u$ and adds the final entry in the second $\bullet$ of (b) with $\ell=s+1$ and $t=u$.
Now $C_{3\cdot2^{s+1}-2}$ has
\begin{itemize}
\item $2^{e-s-2}$ in row $3\cdot2^s-1$;
\item $2^{e-s-2}p_3(x^{2^s})$ in row $2^{s+2}-1$;
\item multiples of $2^{e-s-2}$ in rows 0 through $2^{s+2}-1$;
\item a leading $2^{e-s-1}$ in row $3\cdot2^{s+1}-1$;
\item other entries 0.
\end{itemize}

Now we subtract multiples of row $3\cdot2^s-1$ from all other rows except row $3\cdot2^{s+1}-1$ to make them 0 in column $3\cdot2^{s+1}-2$. By property (d), this will zero all
entries in  column $2^u+2^{s+1}-2$, $u>s+2$, except in rows $3\cdot2^s-1$, $2^{s+2}-1$, and $2^u+2^{s+1}-1$. For $u> s+2$, the entry in $(2^{s+2}-1,2^u+2^{s+1}-2)$ is changed from
$2^{e-u}p_{2^{u-s}-1}(x^{2^s})$ to
\begin{eqnarray*}&&2^{e-u}(p_{2^{u-s}-1}(x^{2^s})-p_3(x^{2^s})p_{2^{u-s-1}-1}(x^{2^{s+1}}))\\
&=& -2^{e-u}x^{2^{s+1}}p_{2^{u-s-1}-2}(x^{2^{s+1}}).\end{eqnarray*}
The minus here can be changed to plus by modifying by a multiple of row $2^u+2^{s+1}-1$, which will not affect the properties such as (d) and (e).
Property (d) will now hold in $N_{s+1}$ for proportionality out of $C_{2^{s+3}+2^{s+1}-2}$, to the extent claimed there.
This change removes the entry of the second $\bullet$ of (b) with $\ell=s$, $m=2^s$,  and $t=u$ and replaces it by the entry of the first $\bullet$ with $\ell=s+1$, $m=0$, and
$t=u$.

Finally we consider the effect of this step on $x$-divisibility. If $j$ is as in (e) with $t>s+1$, and $i\le 2^{s+2}-1$ and $i\ne3\cdot2^s-1$, then
the new value of $P_{i,j}$ will equal
$$P_{i,j}^{\text{old}}-\frac{P_{i,3\cdot2^{s+1}-2}}{2^{e-s-2}}\cdot P_{3\cdot2^s-1,j}.$$
The old $P_{i,j}$ is divisible by $x^{3\cdot2^t-2-j+\eta_s(i)}$. Also, $P_{i,3\cdot2^{s+1}-2}$ is divisible by $x^{\eta_s(i)}$ if $i\le2^{s+2}-2$, and by $x^0$ if $i=2^{s+2}-1$.
(Note that (e) did not apply in this latter case due to the condition there which here would say $i\le j-2^{s+1}$.) We now have $P_{3\cdot2^{s}-1,j}$ divisible by
$x^{3\cdot2^t-2-j}$ since $\eta(3\cdot2^s-1)$ became 0 at the previous substep. Thus the $x$-divisibility of $P_{i,j}$ does not decrease except when $i=2^{s+2}-1$, where it changes
to 0,
consistent with $\eta_{s+1}(2^{s+2}-1)=0$.

\end{proof}

\section{An easily-checked proof for $e\le5$}\label{annsec}

In this section, we give an easily checked proof of Theorem \ref{annthm} for $e\le5$. Its discovery used the reduced form for $M_4$ described in Section \ref{descripsec},
and a {\tt Mathematica} calculation by Gonz\'alez for the $M_5$ analogue. However, checking its validity only requires elementary verifications.

It is proved in \cite[Proposition 4.1]{GVW} that Theorem \ref{annthm} would follow from showing that
\begin{equation}\label{nonz}2^{e-k}u^{3\cdot2^{k-1}-3}[0,0]\ne0\text{ in }M_e\text{ for }1\le k\le e.\end{equation}
For $e\le5$, (\ref{nonz}) is an immediate consequence of the following, which is the main result of this section.

\begin{thm}\label{phi} For $e\ge1$ and  $1\le k\le \min(e,5)$, there is a homomorphism $\phi_{k,e}:M_e\to\Z/2^{k+e-1}$ sending $2^{e-k}u^{3\cdot2^{k-1}-3}[0,0]$ nontrivially.\end{thm}
The homomorphism $\phi_{k,e}$ is nonzero only on the component of $M_e$ in grading $2(3\cdot2^{k-1}-3)$. The component of $M_e$ in  grading $2d$ is generated by the same monomials
$u^{d-i-j}[i,j]$ for any $e$, but the relations depend on $e$.
We will give an explicit formula for $\phi_{k,e}(u^{3\cdot2^{k-1}-3-i-j}[i,j])\in\Z$ for $i,j\ge0$, which is independent of $e$.
Thus we usually call it just $\phi_k$. We will prove that $\phi_k$ applied to a relation (\ref{reln}) in $M_e$ is divisible by $2^{k+e-1}$.
Since part of our formula is $\phi_k(u^{3\cdot2^{k-1}-3}[0,0])=2^{2k-2}$ and hence
$$\phi_k(2^{e-k}u^{3\cdot2^{k-1}-3}[0,0])=2^{k+e-2}\ne0\in\Z/2^{k+e-1},$$ Theorem \ref{phi} will follow.
 The hope was to see a pattern in the formulas for $\phi_k$ that might extend to all $k$,
but they seem a bit too delicate for that.

Since the exponent of $u$ in $u^{3\cdot2^{k-1}-3-i-j}[i,j]$ is determined by $k$, $i$, and $j$, we do not list it. We write $\phi_k(i,j)$ for $\phi_k(u^{3\cdot2^{k-1}-3-i-j}[i,j])$,
and will sometimes omit the subscript $k$.
We have $\phi_1(0,0)=1$, and the only relation in grading 0 in $M_e$ is $2^e[0,0]$, which handles the case $k=1$.

Here are the lists of values of $\phi_k(i,j) $ when $k=2$ and $k=3$.
$$[4\ |\ 0,0\ |\ 2,2,2\ |\ 0,1,1,0],$$
\begin{gather*}[16\ |\ 0,0\ |\ 0,0,0\ |\ 0,0,0,0\ |\ 8,0,8,0,8\ |\ 0,8,0,0,8,0\ |\ 0,0,4,0,4,0,0\ |\\
0,4,4,4,4,4,4,0\ |\ 0,0,6,6,4,6,6,0,0\ |\ 0,0,0,-1,-1,-1,-1,0,0,0]\end{gather*}
Our functions always satisfy $\phi(i,j)=\phi(j,i)$. The first line says that
the nonzero values of $\phi_2$ are $\phi_2(0,0)=4$, $\phi_2(2,0)=\phi_2(1,1)=2$, and $\phi_2(2,1)=1$, and their flips.
The next pair  of lines says, for example, that $\phi_3(0,0)=16$ and
$$\phi_3(i,6-i)=\begin{cases}4&i=2,4\\ 0&i=0,1,3,5,6.\end{cases}$$

Before we list the formulas for $\phi_4$ and $\phi_5$, we discuss the verification that $\phi_{3,e}:M_e\to\Z/2^{e+2}$ is well-defined for all $e\ge3$.
 This one is simple enough that it can be (and was) done by hand. We first consider the case $e=3$.
The coefficients $\binom81,\ldots\binom88$  in (\ref{reln}) are of the form $8,4\a,8\a',2\a'',8\b,4\b',8,1$,
where the $\a$'s are 3 mod 4, and the $\b$'s odd. There are 55 relations after symmetry is taken into account, but only 13 of them contain any
term for which $\nu(\binom8{\ell+1}\phi(i-\ell,j))<5$.  The most delicate is the case $i=5$, $j=4$, in which we have
\begin{eqnarray*}&&8\phi(5,4)+4\a\phi(4,4)+8\a'\phi(3,4)+2\a''\phi(2,4)+8\b\phi(1,4)+4\b'\phi(0,4)\\
&=&8\cdot1-4\a\cdot4+8\a'\cdot4-2\a''\cdot4+8\b\cdot8+4\b'\cdot8\\
&\equiv&8+16+0-24+0+0\equiv0\pmod{32}.\end{eqnarray*}

If $e>3$, then it is as if the binomial coefficients are multiplied by $2^{e-3}$. Their odd factors change, but where it matters, the odd factors are still 3 mod 4.
So $\phi$ applied to each relation is divisible by $2^{e-3}\cdot32$.
Terms with $\binom{2^e}9$ and $\binom{2^e}{10}$ also appear, but they are multiplied by $\phi(0,0)$ or $\phi(0,1)$, and so yield multiples of $2^{e+2}$.
This establishes the well-definedness of $\phi_{3,e}$, and that of $\phi_{2,e}$ is much easier.

Next we list values of $\phi_4(i,j)$ in rows of fixed $i+j$ for which there are some nonzero values. We precede the row by the value of $i+j$. For example, the third listed row
says that
$$\phi_4(i,10-i)=\begin{cases}32&i=2,8\\ 0&i=0,1,3,4,5,6,7,9,10.\end{cases}$$
\begin{gather*}0:64\\
8:32,0,0,0,32,0,0,0,32\\
10:0,0,32,0,0,0,0,0,32,0,0\\
12:0,0,0,0,16,0,0,0,16,0,0,0,0\\
14:0,0,16,0,16,0,16,0,16,0,16,0,16,0,0\\
15:0,0,0,0,0,16,16,0,0,16,16,0,0,0,0,0\\
16:0,0,0,0,8,0,8,0,16,0,8,0,8,0,0,0,0\\
17:0,16,0,16,16,8,0,16,-8,-8,16,0,8,16,16,0,16,0\\
18:0,0,0,0,8,8,4,8,-4,0,-4,8,4,8,8,0,0,0,0\\
19:0,8,8,0,0,-4,-4,-4,4,8,8,4,-4,-4,-4,0,0,8,8,0\\
20:0,0,-4,-4,-4,0,-6,-6,-4,2,4,2,-4,-6,-6,0,-4,-4,-4,0,0\\
21:0,0,0,6,6,6,0,3,3,1,-1,-1,1,3,3,0,6,6,6,0,0,0.
\end{gather*}
These numbers were discovered using Table \ref{init}. Because of the way that they were obtained, it better be the
case that they send all relations to 0, at least when $e=4$. The beauty is that despite the hard work that went into obtaining them, once we
have them, it is a simple computer check to verify that they work. It is just a matter of reading these numbers $\phi_4(i,j)$ into the computer and then having the
computer check that
$$\sum_{\ell=0}^i\tbinom{16}{\ell+1}\phi_4(i-\ell,j)\equiv0\pmod {128}\text{ for }0\le i\le 21,\ 0\le j\le 21-i.$$

Now we can prove by induction on $e$ that if $e>4$, then
$$\sum_{\ell=0}^i\tbinom{2^e}{\ell+1}\phi_4(i-\ell,j)\equiv0\pmod {2^{e+3}}\text{ for }0\le i\le 21,\ 0\le j\le 21-i.$$
It is easy to prove that, for $1<\ell<2^{e+1}$,
\begin{equation}\label{indn}\nu(\tbinom{2^{e+1}}\ell-2\tbinom{2^e}\ell)=2e+1-[\log_2(\ell-1)]-\nu(\ell).\end{equation}
The induction argument follows  from this and the values of $\phi_4(-)$ listed above. Indeed, the induction step requires
\begin{equation}\label{logl}\nu(\phi_4(i-\ell,j))\ge[\log_2(\ell)]+\nu(\ell+1)-1,\end{equation}
and since $i+j\le21$, we have $\nu(\phi_4(i-\ell,j))\ge1$, 2,  3, 4, 5, 6 if $\ell\ge1$, 2, 4,  6, 10, 14, respectively, from which (\ref{logl}) follows.

Our treatment for $\phi_5$ is similar. Because of the longer lists, we take advantage of symmetry, and only list $\phi_5(i,j)$ for $i\le j$.
As before, we list values of $\phi_5(i,j)$ in rows of fixed $i+j$ for which there are some nonzero values. We precede the row by the value of $i+j$.
If $i+j=2t+1$ (resp.~$2t$), the last entry listed is $\phi_5(t,t+1)$ (resp.~$\phi_5(t,t)$).
\begin{eqnarray*}0&:&256\\
16&:&128,0,0,0,0,0,0,0,128\\
20&:&0,0,0,0,128,0,0,0,0,0,0\\
24&:& 0, 0, 0, 0, 0, 0, 0, 0, 64, 0, 0, 0, 0\\
28&:& 0,0,0,0,64,0,0,0,64,0,0,0,64,0,0\\
30&:& 0,0,0,0,0,0,0,0,0,0,64,0,64,0,0,0\\
32&:&0, 0, 0, 0, 0, 0, 0, 0, 32, 0, 0, 0, 32, 0, 0, 0,  64\\
33&:&0, 0, 0, 0, 0, 0, 0, 0, 0, 64, 0, 0, 0, 0, 0, 0, 64\\
34&:&0, 0, 64, 0, 0, 0, 64, 0, 64, 0, -32, 0, 0, 0, 64, 0, 32, 0\\
35&:&0, 0, 0, 0, 0, 0, 0, 0, 0, 64, 64, 64, 0, 64, 64, 0,  64, 0\\
36&:&0, 0, 0, 0, 0, 0, 0, 0, 32, 0, 32, 0, 16, 0, 32, 0, -16,  0, 0\\
37&:&0, 0, 0, 0, 0, 0, 0, 0, 0, 32, 0, 0, 32, 32, 0, 0, 32,  0, 0\\
38&:&0, 0, 32, 0, 32, 0, 0, 0, 0, 0, 16, 0, -16, 32, -16, 0,  16, 32, 0, 0\\
39&:&0, 0, 0, 0, 0, 32, 32, 0, 0, 0, 0, 0, 32, 16, 16, 0, 0,  16, 16, 32\\
40&:&0, 0, 0, 0, 16, 0, 16, 0, 16, 0, 32, 32, 24, 0, 56, 32, 16,  32, 56, 32, 48\\
41&:&0, 32, 0, 32, 32, 48, 0, 0, 16, 48, 0, 48, 48, 56, 0, 16, 8,  24, 16, 16, 8\\
42&:&0, 0, 0, 0, 16, 16, 8, 16, 8, 16, 40, 0, 40, 8, 36, 8, 28,  16, 52, 8, 28, 48\\
43&:&0, 16, 16, 0, 0, 8, 8, 24, 8, 8, 8, 24, 0, 28, 28, 12, 4,  24, 8, 12, 12, 8\\
44&:&0, 0, 24, 24, 24, 0, 28, 28, 16, 4, 28, 24, 0, 0, 18, 2, 4,  26, 28, 2, 20, 2, 20\\
45&:&0, 0, 0, 12, 12, 12, 0, 10, 10, 14, 14, 4, 8, 10, 0, 15, 15,  5, 3, 15, 9, 1, 15
\end{eqnarray*}

The computer checks that
$$\sum_{\ell=0}^i\tbinom{2^e}{\ell+1}\phi_5(i-\ell,j)\equiv0\pmod {2^{e+4}}\text{ for }0\le i\le 45,\ 0\le j\le 45-i$$
is true for $e=5$. It is then proved for all $e\ge5$ by induction, using (\ref{indn}) as in the previous case.

\def\line{\rule{.6in}{.6pt}}

\end{document}